\documentclass[a4paper,11pt]{article}
\usepackage{amsmath,amsthm,amsfonts,amssymb, color}
\usepackage{graphicx}
\usepackage[latin1]{inputenc}
\usepackage[all]{xy}
\usepackage{enumerate}
\usepackage{graphicx}
\usepackage{natded}

\usepackage{amsthm}
\usepackage{float}
\usepackage[bottom]{footmisc}
\usepackage{authblk}

\newtheorem{theorem}{Theorem}[section]
\newtheorem{lem}[theorem]{Lemma}
\newtheorem{cor}[theorem]{Corollary}
\newtheorem{prop}[theorem]{Proposition}
\newtheorem{thm}[theorem]{Theorem}
\newtheorem{ex}[theorem]{Example}

\theoremstyle{definition}
\newtheorem{defn}[theorem]{Definition}
\newtheorem{rem}[theorem]{Remark}

\newcommand{\ra}{\rightarrow}
\newcommand{\Ra}{\Rightarrow}
\newcommand{\rab}{\rightsquigarrow}

\newcommand{\we}{\wedge}
\newcommand{\Fil}{\mathrm{Fil}}
\newcommand{\IFil}{\mathrm{IFil}}

\newcommand{\sq}{\square}

\newcommand{\sha}{\mathbf{sHA}}
\newcommand{\srs}{\mathbf{SRS}}
\newcommand{\shs}{\mathbf{sHL}}

\newcommand{\srl}{\mathbf{SRL}}
\newcommand{\shrl}{\mathbf{SHL}}

\newcommand{\srlbs}{\mathbf{SRLbs}}

\title{On subreducts of subresiduated lattices and logic}

\author[1, 2]{\small Jos\'e Luis Castiglioni\footnote{Corresponding author: e-mail: jlc@mate.unlp.edu.ar; tel.: +54 221 4245875}}
\author[3]{\small V\'{\i}ctor Fern\'andez}
\author[3, 2]{\small H\'ector Federico Mallea}
\author[1, 2]{\small Hern\'an Javier San Mart\'{\i}n}

\affil[1]{\footnotesize Departamento de Matem\'atica, Facultad de Ciencias Exactas (UNLP). Casilla de correos 172, La Plata (1900), Argentina.}
\affil[2]{\footnotesize CONICET-Argentina}
\affil[3]{\footnotesize Universidad Nacional de San Juan. Avenida Jos\'e I. de la Roza (Oeste) 230, San Juan (5400), Argentina.}

\date{}

\begin{document}

\maketitle

\begin{abstract}
	Subresiduated lattices were introduced during the decade of 1970 by Epstein and Horn as an algebraic counterpart of some logics with strong implication previously studied by Lewy and Hacking.  These logics are examples of subuintuitionistic logics, i.e., logics in the language of intuitionistic logic that are defined semantically by using Kripke models, in the same way as intuitionistic logic is defined, but without requiring of the models some of the properties required in the intuitionistic case. Also in relation with the study of subintuitionistic logics, Celani and Jansana get these algebras as the elements of a subvariety of that of weak Heyting algebras.
	
	Here, we study both the implicative and the implicative-infimum subreducts of subresiduated lattices. Besides, we propose a calculus whose algebraic semantics is given by these classes of algebras. Several expansions of this calculi are also studied together to some interesting properties of them.
\end{abstract}

\section{Introduction}
\label{intro}

Subresiduated lattices were introduced in \cite{EH}.
A subresiduated lattice is a pair $(L,D)$, where $L$ is a bounded distributive lattice and
$D$ is a bounded sublattice of $L$ such that for every $a,b\in L$ the set
\[
E_{ab} := \{d \in D \ | \ d \wedge a \leq b\}
\]
has maximum element. In this case we define the binary operation $\ra$ by
\[
a \rightarrow b := \max E_{ab}.
\]
Note that if $(L,D)$ is a subresiduated lattice and $D = L$ then the operation
$\ra$ is the residuum of $\we$, i.e., $(L, \ra)$ is a Heyting algebra. Moreover,
subresiduated lattices $(L,D)$ can be seen as algebras $(L,\we,\vee,\ra,0,1)$ of type
$(2,2,2,0,0)$, where $D = \{a\in A:1\ra a = a\} = \{1\ra a:a\in A\}$.
The class of subresiduated lattices forms a variety, which will be denoted by $\srl$ \cite{EH}.

It is known that the class of $\{\ra,1\}$-subreducts of Heyting algebras is the variety of
Hilbert algebras, which was studied by Diego in $\cite{D}$. It naturally arises the following question:
can we characterize the class of $\{\ra\,1\}$-subreducts of the elements in the variety $\srl$? The answer of this question is the first goal of this article. Note that this class contains as subclass that of Hilbert algebras. Since in every subresiduated lattice
the equation $x\ra x = 1$ is satisfied, the class of $\{\ra\}$-subreducts of the members of $\srl$ and the class
of the $\{\ra,1\}$-subreducts of $\srl$ are term equivalent, so in order to make
the exposition of this paper a bit simpler we will speak of $\{\ra\}$-subreducts in place of $\{\ra,1\}$-subreducts.   

We also show that the class of $\{\ra\}$-subreducts of the members of $\srl$ is a quasivariety which is not a variety and we present a quasi-equational base for it. We named sub-Hilbert algebras the elements of this quasivariety.

Another approach to subresiduated lattices was proposed in \cite{CJa}. There, these are seen as the elements ot the subvariety of weakly Heyting algebras, which are related to the algebraic counterpart of strict implication fragment of the global consequence relation of the normal modal logic K \cite{CJa}. The interest in these algebras come from the study of subintuitionistic logics, intended as those logics in the language of intuitionistic logic that are defined semantically by using Kripke models in the
same way as intuitionistic logic is, but without requiring of the models some of the properties required in the intuitionistic case \cite{CJl}.

The study of a propositional logic whose algebraic semantic is the class of sub-Hilbert algebras is also carry out. This is another goal of this article. This logic is close related with the implicational fragment of the logic R4 of \cite{EH}. Besides, we investigate several expansions of this logic within the language of intuitionistic logic.
\vskip.2cm

The structure of this article is as follows. In Section \ref{srls} we recall the definition of subresiduated lattices and the main properties of these algebras which will be needed in subsequent sections. As was already mentioned, Section \ref{subHilbert} is devoted to the study of the $\{\ra\}$-subreducts of subresiduated lattices. Using similar techniques, we study the $\{\wedge, \ra\}$-subreducts of subresiduated lattices in Section \ref{semireticulos}. Section \ref{logic} is the lengthier one. There, we study some logics associated with the algebras studied in previous sections. Through a Hilbert style system, in Subsection \ref{subHilbertLogic} we present a logic $L_1$ in the language $\{\ra\}$ which has as algebraic semantics the quasivariety studied in Section \ref{subHilbert}. Axiomatic expansions of $L_1$ with the usual intuitionistic connectives are studied in Subsection \ref{expansionsofL1}. Here, we compare our logic in the full intuitionistic signature with that introduced in \cite{EH}. Motivated by logic, we also propose an extension of the notion of subresiduated lattices, by avoiding the distributive condition on its underlying lattice structure.  Finally, in Subsection \ref{otherexpansion}, we expand the system $L_1$ with a weaker conjunction and characterize the variety of its algebraic semantics.

\section{Subresiduated lattices}
\label{srls}

As we have mentioned in the previous section, the class of subresiduated lattices forms a variety \cite{EH,CJa}.
In what follows we give an equational presentation for this variety \cite[Theorem 1]{EH}.

\begin{defn}\label{srl}
	A \textbf{subresiduated lattice} is an algebra $(A,\we,\vee,\ra,0,1)$ of type $(2,2,2,0,0)$,
	where $(A,\we,\vee,0,1)$ is a bounded distributive lattice and the following equations
	are satisfied:
	\begin{enumerate}
		\item[\textbf{(A1)}] $x\ra x= 1$,
		\item[\textbf{(A2)}] $x \ra y \leq z \ra (x \ra y)$,
		\item[\textbf{(A3)}] $x \we (x \ra y) \leq y$
		\item[\textbf{(A4)}] $z \ra (x \we y) = (z \ra x) \we (z \ra y)$,
		\item[\textbf{(A5)}] $(x \vee y) \ra z = (x \ra z) \we (y \ra z)$,
		\item[\textbf{(A6)}] $(x \ra y) \we (y \ra z) \leq (x \ra z)$.
	\end{enumerate}
\end{defn}

If $(A,\we,\vee,\ra,0,1)$ is a subresiduated lattice then
for brevity the notation $(A,\ra)$ will be used, assuming that $A$ is a bounded distributive lattice.

If $(A,D)$ is a subresiduated lattice in the sense of the definition given in the introduction
then $(A,\ra)$ is a subresiduated lattice in the sense of Definition \ref{srl}. Moreover,
$D = \{a\in A: 1\ra a = a\} = \{1\ra a:a\in A\}$.
Conversely, if $(A,\ra)$ is a subresiduated lattice
in the sense of Definition \ref{srl} and we define $D = \{a\in A:1\ra a = a\}$, then the pair
$(A,D)$ is a subresiduated lattice in the sense of the definition given in the introduction.

If $a$ is an element of a subresiduated lattice $(A,\ra)$ then we define
$\square(a): = 1\ra a$ and $\sq A = \{\square(a):a\in A\}$.

\begin{rem}
	Note that if $(A,D)$ is a subresiduated lattice then $(D,\ra)$ is a Heyting algebra which
	is a subalgebra of the subresiduated lattice $A$.
	Moreover, for every $a, b \in A$ and $d \in \sq A$ we have that $a \we d \leq b$ if and only if $d \leq a \ra b$.
	Also note that if we do not assume the condition $d \in \sq A$, in general we only have that if $d \leq a\ra b$
	then $a\we d \leq b$.
\end{rem}

The proof of the following two lemmas are straightforward.

\begin{lem}\label{BAI}
	If $(A,\ra)$ is a subresiduated lattice and $a, b \in A$ then $a \leq b$ if and only if $a \to b = 1$.
	Moreover, the following cuasi-equations are satisfied in every subresiduated lattice:
	\begin{enumerate}
		\item[\textbf{(I)}] $x \to x = 1$,
		\item[\textbf{(T)}] $x \to 1 = 1$,
		\item[\textbf{(A)}] if $x \to y = 1$ and $y \to x = 1$ then $x = y$,
		\item[\textbf{(B)}] $(x \ra y) \ra ((y \ra z) \ra (x \ra z)) = 1$.
	\end{enumerate}
\end{lem}

\begin{lem}\label{sq-props}
	Let $(A,\ra)$ be a subresiduated lattice. Then
	$\sq(x) \leq x$ for every $x\in A$.
\end{lem}

\begin{rem}
	Let $(A,\ra)$ be a subresiduated lattice and $x,y\in A$.
	Since $x\ra y \in \sq A$ then
	\begin{equation}\label{sq^2}
		\sq(x \ra y) = x \ra y.
	\end{equation}
	In particular, $\sq^2 x = \sq x$ for every $x\in A$
	(this fact also follows from that $\sq x\in \sq A$ for
	every $x\in A$).
\end{rem}

\begin{lem}\label{sr1-2}
	In every subresiduated lattice the following equations are satisfied:
	\begin{enumerate}
		\item[\textbf{(S1)}] $(x \ra y) \ra (z \ra (x \ra y)) = 1$,
		\item[\textbf{(S2)}] $w \ra (x \ra (y \ra z)) \ra \left[(w \ra (x \ra y)) \ra (w \ra (x \ra z))\right] = 1$.
	\end{enumerate}
\end{lem}

\begin{proof}
	The equation \textbf{(S1)} is equivalent to the equation \textbf{(A2)}.
	
	The equation \textbf{(S2)} is equivalent to prove that
	\[
	(w \ra (x \ra (y \ra z))) \we (w \ra (x \ra y)) \leq w \ra (x \ra z).
	\]
	Note that
	\[
	(w \ra (x \ra (y \ra z))) \we (w \ra (x \ra y)) = w\ra ((x \ra (y \ra z))\we (x \ra y))
	\]
	and
	\[
	w\ra ((x \ra (y \ra z))\we (x \ra y)) = w\ra (x\ra (y\we (y\ra z))).
	\]
	
	Then \textbf{(S2)} is satisfied if and only if
	\begin{equation}\label{eqs2}
		w\ra (x\ra (y\we (y\ra z))) \leq w \ra (x \ra z).
	\end{equation}
	Since $y\we (y\ra z)) \leq z$ then (\ref{eqs2}) is satisfied, which was our aim.
\end{proof}

\begin{rem}\label{sobreS}
	Note that \textbf{(S1)} is also consequence of \eqref{sq^2} and the antimonotony of $\ra$ in the first coordinate, which is consequence of \textbf{(A5)}.
	Indeed, since $z \leq 1$ then $1 \ra (x \ra y) \leq z \ra (x \ra y)$.
	Taking into account that $1 \ra (x \ra y) = \sq(x \ra y) = x \ra y$ we obtain \textbf{(S1)}.
	
	Besides, note that taking $w = 1$ in \textbf{(S2)} we have that	
	\[
	1 \ra (x \ra (y \ra z)) \ra \left[(1 \ra (x \ra y)) \ra (1 \ra (x \ra z))\right] = 1.
	\]
	Thus, by \eqref{sq^2} we deduce
	\begin{enumerate}
		\item[\textbf{(S)}] $(x \ra (y \ra z)) \ra \left[(x \ra y) \ra (x \ra z)\right] = 1$.
	\end{enumerate}
	
	On the other hand, assuming \textbf{(S)} we have that
	\[
	x \ra (y \ra z) \leq \left[(x \ra y) \ra (x \ra z)\right].
	\]
	By the monotony of $\ra$ in the second coordinate, which follows from \textbf{(A4)}, we get
	\[
	w \ra (x \ra (y \ra z)) \leq w \ra \left[(x \ra y) \ra (x \ra z)\right].
	\]
	Applying again \textbf{(S)},
	\[
	w \ra \left[(x \ra y) \ra (x \ra z)\right] \leq \left[ w \ra (x \ra y)\right] \ra \left[ w \ra(x \ra z)\right],
	\]
	from where, by transitivity, it follows that
	\[
	w \ra (x \ra (y \ra z)) \leq \left[ w \ra (x \ra y)\right] \ra \left[ w \ra(x \ra z)\right],
	\]
	i.e., \textbf{(S2)}.
\end{rem}

\section{On the $\{\ra\}$-subreducts of subresiduated lattices}
\label{subHilbert}

We start this section proving that the class of $\{\ra\}$-subreducts\footnote{In this paper we take as subreduct of an algebra $A$, any isomorphic image of a subalgebra of the corresponding reduct of $A$. With this definition, the class of subreducts of a class of algebras is always closed by subalgebras and isomorphic images.}
of subresiduated lattices do not form a variety.

\begin{lem}\label{noesvariedad}
	The class of $\{\ra\}$-subreducts of subresiduated lattices\footnote{Note that any $\{\ra\}$-subreduct of a subresiduated lattice always contains the constant $1 = x \ra x$, and hence may be seen as a $\{\ra, 1\}$-subreduct of it; i.e., as an algebra of type (2,0).} is not a variety.
\end{lem}

\begin{proof}
	Consider the $\{\ra\}$-reduct of the subresiduated lattice $(L, D)$, where $L$
	is the chain of three elements $0 < m < 1$ and $D = \{0, 1\}$.
	Let $(A, \ra, a)$ be the algebra of type $(2,0)$ with universe the
	set of two elements $A = \{a, b\}$
	such that for every $x, y \in A$, $x \ra y = a$.
	Let $f : L \to A$ be the map given by $f(0) = f(1) = a$ and $f(m) = b$.
	It is immediate that $f$ is an homomorphism of algebras of type $(2, 0)$
	and that $A$ is not a $\{\ra\}$-subreduct of a subresiduated lattice.
	Indeed, $a \neq b$ although $a \ra b = b \ra a = f(1)$.
	Thus, the class of $\{\ra\}$-subreducts of subresiduated lattices
	is not closed by homomorphic images.
\end{proof}

We write $\mathcal{K}$ for the class of $\{\ra\}$-subreducts of subresiduated lattices.
The aim of this section is to show that $\mathcal{K}$ is a quasi-variety.
In order to prove it we introduce in what follows a quasi-variety which contains $\mathcal{K}$. Then, we see that this quasi-variety is contained in $\mathcal{K}$.

\begin{defn}\label{sha}
	A \textbf{sub-Hilbert algebra} is an algebra $(A, \ra, 1)$ of type (2,0) which satisfies the following quasi-equations:
	\begin{enumerate}
		\item[\textbf{(B)}] $(x \ra y) \ra ((y \ra z) \ra (x \ra z)) = 1$,
		\item[\textbf{(I)}] $x \to x = 1$,
		\item[\textbf{(T)}] $x \to 1 = 1$,
		\item[\textbf{(A)}] if $x \to y = 1$ and $y \to x = 1$ then $x = y$,
		\item[\textbf{(S)}] $(x \ra (y \ra z)) \ra ((x \ra y) \ra (x \ra z)) = 1$.
	\end{enumerate}
\end{defn}

In what follows we write $\sha$ to indicate the class of sub-Hilbert algebras.
As in the case of subresiduated lattices, if $(A,\ra,1)$ is a sub-Hilbert algebra
and $a\in A$, we define $\sq a = 1\ra a$ and $\sq A = \{\sq a \ | \ a \in A\}$.
\vskip.2cm

Note that in every algebra $(A,\ra,1)$ of type $(2,0)$ which verifies the cuasi-equations
$\textbf{(B)}$, $\textbf{(I)}$, $\textbf{(A)}$ and $\textbf{(T)}$ of the Definition \ref{sha}, the relation $\leq$, given by $a\leq b$ if and only if $a\ra b = 1$, is an order. This order will be referred in what follows as ``the natural order given by the implication''. Moreover, $1$ is the last element with respect to this order.
\vskip.2cm

The following properties of monotony and antimonotony of the operation $\ra$ in the algebras of $\sha$ will be useful later.

\begin{lem}\label{monotonia}
	Let $A \in \sha$ and $a, b, c \in A$. Then
	\begin{enumerate}
		\item[M1.] If $a \leq b$ then $b \ra c \leq a \ra c$,
		\item[M2.] If $a \leq b$ then $c \ra a \leq c \ra b$.
	\end{enumerate}
\end{lem}

\begin{proof}
	First, we show $M1$. Suppose that $a \leq b$. Then
	\begin{equation}\label{e1}
		a \ra b = 1.
	\end{equation}
	It follows from \textbf{(B)} that
	\begin{equation}\label{e2}
		(a \ra b) \ra ((b \ra c) \ra (a \ra c)) = 1.
	\end{equation}
	It follows from \eqref{e1} and \eqref{e2} that
	\begin{equation}\label{e2-1}
		1 \ra ((b \ra c) \ra (a \ra c)) = 1.
	\end{equation}
	By \textbf{(T)} we have that
	\begin{equation}\label{e2-2}
		((b \ra c) \ra (a \ra c)) \ra 1 = 1.
	\end{equation}
	Thus, by (\ref{e2-1}), (\ref{e2-2}) and \textbf{(A)} we get $b \ra c \leq a \ra c$.
	
	Finally we will see $M2$. Suppose that $a\leq b$, i.e., $a\ra b = 1$.
	It follows from \textbf{(S)} that
	$(c \ra (a\ra b)) \ra ((c\ra a)\ra (c\ra b)) = 1$.
	By \textbf{(T)}, $c \ra (a\ra b) = c\ra 1 = 1$, so
	$1\ra ((c\ra a)\ra (c\ra b)) = 1$.
	On the other hand, using  \textbf{(T)} again, we get $((c\ra a)\ra (c\ra b)) \ra 1 = 1$.
	Therefore, it follows from \textbf{(A)} that $(c\ra a)\ra (c\ra b) = 1$,
	i.e., $c\ra a \leq c\ra b$.
\end{proof}

\begin{lem}\label{saleA2}
	Let $A \in \sha$ and $a, b, c \in A$.
	Then $a \ra b \leq c \ra (a \ra b)$, which is $\textbf{(A2)}$ of Definition \ref{srl}.
\end{lem}

\begin{proof}
	Let $a, b, c \in A$. Since $c\leq 1$ then
	it follows from Lemma \ref{monotonia} that
	\begin{equation}\label{aux1}
		1 \ra (a \ra b) \leq c \ra (a \ra b).
	\end{equation}
	Thus, by $\textbf{(B)}$ and $\textbf{(I)}$,
	\begin{equation}\label{aux2}
		a \ra b \leq (b \ra b) \ra (a \ra b) = 1 \ra (a \ra b).
	\end{equation}
	Hence, taking into account \eqref{aux1} and \eqref{aux2} we get
	$a \ra b \leq c \ra (a \ra b)$.	
\end{proof}

As in $\srl$, we define $\Box x := 1 \ra x$, and $\Box A := \{\Box x: x \in A\}$.

\begin{lem}\label{K}
	Let $A \in \sha$ and $x \in A$. Then $\sq x \leq x$.
\end{lem}

\begin{proof}
	Let $x\in A$. It follows from (\textbf{S}) that
	\[
	\sq x \ra (1 \ra x) \leq (\sq x \ra 1) \ra (\sq x \ra x).
	\]
	Since $\sq x \ra (1 \ra x) = \sq x \ra \sq x = 1$ then $\sq x \ra 1 \leq \sq x \ra x$.
	But $\sq x \ra 1 = 1$, so $\sq x \ra x = 1$, i.e., $\sq x \leq x$.
\end{proof}

Note that it follows from Lemma \ref{saleA2} that in every sub-Hilbert algebra
the condition
\begin{equation}\label{desig1desq}
	x \ra y \leq \sq(x \ra y)
\end{equation}
is satisfied. This fact and Lemma \ref{K} imply that in every sub-Hilbert algebra
the following equation is satisfied:
\begin{equation}\label{igualdaddesq}
	x \ra y = \sq(x \ra y).
\end{equation}

Let $(A, \ra, 1) \in \sha$. In this algebra the condition
$\textbf{(S1)}$ is satisfied (see Remark \ref{sobreS}).
By considering $z = x = 1$ in $\textbf{(S1)}$ we have that
\begin{equation}\label{aux3}
	\sq y \leq \sq^2 y.
\end{equation}
Thus, it follows from Lemma \ref{K} and equation \eqref{aux3}
that for every $y \in A$, $\sq^2 y = \sq y$.

As an immediate consequence, we get  another characterization of $\Box A$: $$\Box A = \{x: \Box x = x\}.$$

\begin{rem}\label{sqAesHilbert}
	It follows from \cite[Definition 1]{D} that every algebra $(A, \ra, 1)$ of type (2,0)
	which satisfies $\textbf{(A)}$, $\textbf{(S)}$ and $\textbf{(h1)}$ is a Hilbert algebra,
	where
	\vskip.2cm
	\noindent $\textbf{(h1)}$ $x \ra (y \ra x) = 1$.
\end{rem}

It follows from equation \eqref{igualdaddesq} that the set $\sq A$ is closed by $\ra$.
Since $1 \in \sq A$ then $(\sq A, \ra, 1)$ is a subalgebra of the sub-Hilbert algebra $(A, \ra, 1)$.
Besides, taking into account $\textbf{(S1)}$, we have that
\[
1 \ra x \leq \sq y \ra (1 \ra x),
\]
i.e., $\sq x \ra (\sq y \ra \sq x) = 1$, which is exactly the equation $\textbf{(h1)}$
over the elements of $\sq A$. Hence, $(\sq A, \ra, 1)$ is a Hilbert algebra.

We have proved the following result.

\begin{lem}\label{sqAesHil}
	Let $(A, \ra, 1) \in \sha$. Then $\sq A$ is the universe
	of a subalgebra of $A$, which is a Hilbert algebra.
\end{lem}

\begin{cor}\label{sqpropiedades}
	Let $A \in \sha$ and $x, y, z \in A$. Then,
	\begin{enumerate}
		\item $\sq x \ra (\sq y \ra \sq z) = (\sq x \ra \sq y) \ra (\sq x \ra \sq z)$,
		\item $(\sq x \ra (\sq y \ra z)) \ra (\sq y \ra (\sq x \ra z)) = 1$.
	\end{enumerate}
\end{cor}

\begin{proof}
	We only need to show 2. Since $\sq A$ is a Hilbert algebra we have that for every $x,y,z\in A$,
	\begin{equation}\label{paux4}
		\sq x \ra (\sq y \ra \sq z) = \sq y \ra (\sq x \ra \sq z).
	\end{equation}	
	Besides, for every $y, z \in A$ we have, by monotony of $\ra$ in the first coordinate and the inequality $\sq z  \leq z$,
	that $\sq y \ra \sq z \leq \sq y \ra z$.
	We also have that
	\[
	\sq y \ra z = \sq (\sq y \ra z) = 1 \ra (\sq y \ra z) \leq \sq^2 y \ra \sq z = \sq y \ra \sq z.
	\]
	Then,
	\begin{equation}\label{aux4}
		\sq y \ra \sq z = \sq y \ra z.
	\end{equation}
	Therefore, it follows from  \eqref{aux4} that
	\[
	\sq x \ra (\sq y \ra \sq z) = \sq x \ra (\sq y \ra z),
	\]
	Interchanging the roles of $x$ and $y$ in previous computations we get the following equality,
	\[
	\sq y \ra (\sq x \ra \sq z) = \sq y \ra (\sq x \ra z).
	\]
	Using these equalities in \eqref{paux4}, we get
	\[
	\sq x \ra (\sq y \ra  z)  = \sq y \ra (\sq x \ra z),
	\]
	which was our aim.
\end{proof}	

\begin{prop}\label{KessHL}
	$\mathcal{K} \subseteq \sha$.
\end{prop}

\begin{proof}
	It follows from Lemma \ref{BAI} and Remark \ref{sobreS}.
\end{proof}

In the rest of this section we will give some results in order to show
that $\sha \subseteq$ $\mathcal{K}$, which will imply that
$\sha = \mathcal{K}$. In order to make it possible we will
adapt some arguments used in \cite{CFMSM}.
\vskip.2cm

\begin{defn}
	Let $(X,\leq)$ be a poset. A subset $U$ of $X$ is said to be an upset if for every $x,y\in X$,
	if $x\leq y$ and $x\in U$ then $y\in U$. We write $X^+$ for the complete lattice of upsets of $X$.
\end{defn}

In what follows we adapt the definition of implicative filter for sub-Hilbert algebras.

\begin{defn}\label{ifildef}
	Let $(A, \ra)$ be a subresiduated lattice and $F \subseteq A$.
	We say that $F$ is an \textbf{implicative filter} of $A$ if the following
	conditions are satisfied:
	\begin{enumerate}
		\item $1 \in F$.
		\item For every $a,b\in A$, if $a, a \ra b \in F$ then $b \in F$.
	\end{enumerate}
\end{defn}

\begin{rem}\label{filesifil}
	Note that in subresiduated lattices we have that every filter is an implicative filter.
	In particular, every principal filter is an implicative filter.
\end{rem}

We have that the set of implicative filters of a sub-Hilbert algebra $A$ form a complete lattice, which will be denoted by $\IFil(A)$.
We define the map $j_{A} : A \ra \IFil(A)^+$ by
\[
j_{A}(a) := \{F\in \IFil(A) \ | \ a\in F\}.
\]
If there is not ambiguity we write $j$ in place of $j_A$.

\begin{lem}\label{il1}
	Let $A\in \sha$. Then $j$ is an injective map.
\end{lem}

\begin{proof}
	Let $A\in \sha$ and $a, b \in A$. Suppose that $a \neq b$.
	Without loss of generality we can assume that $a \nleq b$.
	Since $\uparrow a \in \IFil(A)$ then $j(a) \neq j(b)$ because
	$\uparrow a \in j(a)$ and $\uparrow a \notin j(b)$. Thus, $j$ is an injective map.
\end{proof}

\begin{rem}
	Note that since $a\leq b$ if and only if $j(a) \subseteq j(b)$,
	we have that $j$ is an order embedding.
\end{rem}

Following the notation employed in \cite{CFMSM}, for every $a, x_1, \cdots, x_{n+1} \in A$
we define recursively
\begin{equation}
	\label{corchete}
	\begin{array}{cc}
		&[x_1, a] = x_1 \ra a,\ \ \ \ \ \ \ \ \ \ \ \ \ \ \ \ \ \ \ \ \ \ \ \ \ \ \ \ \ \ \  \\
		&[x_{n+1}, \dots, x_1, a] = x_{n+1} \ra [x_n, \dots, x_1, a]	
	\end{array}
\end{equation}

For example, $[x,y,z,a]$ is $[x,[y,[z,a]]] = x \ra (y \ra (z \ra a))$.

\begin{lem} \label{permutacion}
	Let $A\in \sha$, $x_1, \dots, x_n \in \sq A$, $a \in A$ and $\sigma$ a permutation of $\{1, \dots, n\}$.
	Then,
	\[
	[x_n, \dots, x_1, a] = [x_{\sigma(n)}, \dots, x_{\sigma(1)}, a].
	\]
\end{lem}

\begin{proof}
	It follows by induction taking into account Corollary \ref{sqpropiedades}.2.
\end{proof}

\begin{lem} \label{monotoniadelbracket}
	Let $A\in \sha$ and $x_1, \dots, x_n, a, b \in A$.
	If $a \leq b$ then
	\[
	[x_n, \dots, x_1, a] \leq [x_n, \dots, x_1, b].
	\]
\end{lem}

\begin{proof}
	It follows by induction taking into account M2.
\end{proof}

\begin{lem} \label{propdelshrlbracket}
	Let $A\in \sha$ and $x_1, \dots, x_n, a, b, c, d \in A$. Then
	\[
	[x_n, \dots, x_1, a, b, c, d] \leq [x_n, \dots, x_1, [a, b, c], a, b, d].
	\]
\end{lem}

\begin{proof}
	We define $\alpha = [a, b, c, d] = a \ra (b \ra (c \ra d))$
	and $\beta = (a \ra (b \ra c)) \ra (a \ra (b \ra d)) = [[a,b,c], a, b, d]$.
	Using $\mathbf{(S)}$ and monotonicity, we get that $\alpha \leq \beta$. Then, by Lemma \ref{monotoniadelbracket} we have that
	\[
	[x_n, \dots, x_1, \alpha] \leq [x_n, \dots, x_1, \beta].
	\]
	Thus,
	\[
	\begin{array}
		[c]{lllll}
		[x_n, \dots, x_1, a, b, c, d]     & =       & [x_n, \dots, x_1, [a, b, c, d]] &  & \\
		& \leq    & [x_n, \dots, x_1, [[a,b,c], a, b, d]]&  &\\
		& =       & [x_n, \dots, x_1, [a,b,c], a, b, d]. &  &
	\end{array}
	\]	
\end{proof}

\begin{lem} \label{distributividad}
	Let $A\in \sha$ and $x_1, \dots, x_n, a, b \in \sq A$. Then
	\[
	[x_n, \dots, x_1, a, b] = [[x_n, \dots, x_1, a], x_n, \dots, x_1, b].
	\]
\end{lem}

\begin{proof}
	It follows by induction taking into account Corollary \ref{sqpropiedades}.1.
\end{proof}

Let $A\in \sha$. If $X \subseteq A$ we will write $\langle X \rangle$ to indicate the
implicative filter generated by $X$.

\begin{lem} \label{il2}
	Let $A\in \sha$, $a \in A$ and $X \subseteq \sq A$ such that $1 \in X$. Then,
	\[
	\langle X \cup \{a\} \rangle = \{b\in A \; | \; \exists \;x_1,\dots,x_n \in X\;
	\text{s.th.}\; [x_1, \dots, x_n, a, b] = 1\}.
	\]
\end{lem}

\begin{proof}
	Let $A\in \sha$, $a \in A$ and $X \subseteq \sq A$ with $1 \in X$.
	We define
	\[
	G := \{b\in A \; | \; \exists \;x_1,\dots,x_n \in X\;\text{s.th.}\; [x_1, \dots, x_n, a, b] = 1\}.
	\]
	First we will see that  $X \cup \{a\} \subseteq G$. Let $x \in X$. Since
	$a \leq 1$, $1 \ra x \leq a \ra x$. However, $1 \ra x = x$, so $x \leq a \ra x$,
	i.e., $x\ra (a\ra x) = 1$. Hence, $x\in G$. Besides, $1 \ra (a\ra a) = 1$ and
	$1\in X$, so $a\in G$. Thus, $X \cup \{a\} \subseteq G$.
	
	Now we will see that $G$ is an implicative filter. Since $1 \ra (a\ra 1) = 1$
	then $1 \in G$. Let $b, c \in A$ such that $b, b \ra c \in G$.
	Then there exist $x_1,\dots, x_n, y_1,\dots, y_m \in X$  such that
	\begin{equation}\label{eqx}
		[x_1, \dots, x_n, a, b] = 1
	\end{equation}
	and
	\begin{equation} \label{eqy}
		[y_1, \dots, y_m, a, b, c] = 1.
	\end{equation}
	Then
	\[
	\begin{array}
		[c]{lllll}
		1    &  =    & [x_1, \dots, x_n, y_1, \dots, y_m, a, b, c] &  & \text{Eq. (\ref{eqy})}\\
		&  =    & [x_1, \dots, x_n, y_1, \dots, y_m, [a, b, c]] &  & \\
		&  =    & [y_1, \dots, y_m, x_1, \dots, x_n, [a, b, c]] &  &\text{Lemma \ref{permutacion}} \\
		&  =    & [y_1, \dots, y_m, x_1, \dots, x_n, a, b, c] &  &\\
		&  \leq & [y_1, \dots, y_m, x_1, \dots, x_{n-1},[x_n, a, b], x_n, a, c] &  & \text{Lemma \ref{propdelshrlbracket}}\\
		&  =    & [y_1, \dots, y_m, x_1, \dots, x_{n-1},[x_n, a, b], [x_n, a, c]] &  &\\
		&  =    & [y_1, \dots, y_m,[x_1, \dots, x_{n-1}, [x_n, a, b]], [x_1, \dots, x_{n-1}, [x_n, a, c]]] &  & \text{Lemma \ref{distributividad}}\\
		&  =    & [y_1, \dots, y_m,[x_1, \dots, x_{n-1}, x_n, a, b], [x_1, \dots, x_{n-1}, x_n, a, c]] &  &\\
		&  =    & [y_1, \dots, y_m, 1, [x_1, \dots, x_{n-1}, x_n, a, c]] &  & \text{Eq. (\ref{eqx})}\\
		&  =    & [y_1, \dots, y_m, x_1, \dots, x_{n-1}, x_n, a, c] &  &\\
	\end{array}
	\]
	Thus,
	\[
	[y_1, \dots, y_m, x_1, \dots, x_{n-1}, x_n, a, c] = 1,
	\]
	so $c\in G$. Then $G$ is an implicative filter.
	
	Finally, it is immediate that if $H$ is an implicative filter such that $X \cup \{a\} \subseteq H$ then
	$G \subseteq H$. Therefore, $G = \langle X \cup \{a\} \rangle$.
\end{proof}

\begin{lem} \label{il3}
	Let $A\in \sha$, $F \in \IFil(A)$ and $a, b \in A$ with $a \ra b \notin F$.
	Then there exists $G \in \IFil(A)$ such that $a \in G$, $b \notin G$ and $F \cap \sq{A} \subseteq G$.
\end{lem}

\begin{proof}
	We define $G$ as the implicative filter generated by $(F\cap \sq A) \cup \{a\}$.
	Then $a\in G$ and $F\cap \sq A \subseteq G$.
	
	Suppose that $b \in G$. It follows from Lemma \ref{il2} that there exist $x_1,\dots,x_n \in F\cap \sq A$
	such that
	\[
	[x_{1}, x_{2}, \cdots,x_{n}, a, b] = 1.
	\]
	Since $x_1,\dots,x_n\in F$ and $F\in \IFil(A)$ then $a\ra b\in F$,
	which is a contradiction. Therefore, $b\notin G$.
\end{proof}

Let $A\in \sha$. We write $D$ for the complete sublattice of $\IFil(A)^+$ generated by $j(\sq A)$.
Note that this sublattice is necessarily distributive.

\begin{lem} \label{il4}
	Let $A\in \sha$,
	$F, G\in \IFil(A)$ and $W\in D$ such that $F\in W$
	and $F\cap \sq A \subseteq G$. Then $G\in W$.
\end{lem}

\begin{proof}
	Let $A\in \sha$, and consider $F, G\in \IFil(A)$ and $W\in D$ such that $F\in W$
	and $F\cap \sq A \subseteq G$. Since $W\in D$ then there exists $\{a_{ik}\}_{i\in I, k\in K} \subseteq \sq A$
	such that
	\[
	W = \bigcup_{i\in I}\bigcap_{k\in K} j(a_{ik}).
	\]
	Taking into account that $F\in W$ and  $\{a_{ik}\}_{i\in I, k\in K} \subseteq \sq A$
	we have that there exists $l \in I$ such that for every $k\in K$, $a_{lk} \in F \cap \sq A$.
	But $F \cap \sq A \subseteq G$, so $a_{lk}\in G$ for every $k\in K$. Hence, $G\in W$.
\end{proof}

Let $A\in \sha$. For every $U,V \in \IFil(A)^+$ there exists the maximum of the set
$\mathcal{B} = \{W\in D: W\cap U\subseteq V\}$.
Indeed, since $\mathcal{B} \subseteq D$ then by definition of $D$, there exists
the supremum of $\mathcal{B}$, which will be called $\alpha$, i.e.,
$\alpha = \bigcup_{W\in \mathcal{B}} W$. Thus, $\alpha\in D$ and
$\alpha \cap U = \bigcup_{W\in \mathcal{B}} (W\cap U) \subseteq V$.
Hence, $\alpha$ is the maximum of $\mathcal{B}$. This maximum will be denoted by
$U\Ra V$. Moreover, $(\IFil(A)^+,D)$ is a subresiduated lattice.

\begin{prop} \label{pp}
	Let $A\in \sha$ and $a,b\in A$. Then
	\[
	j(a\ra b) = j(a) \Ra j(b).
	\]
\end{prop}

\begin{proof}
	Let $A\in \sha$ and $a,b\in A$. We define the set
	\[
	E_{ab} = \{W\in D: W\cap j(a) \subseteq j(b)\}.
	\]
	It is immediate that $j(a\ra b)\cap j(a) \subseteq j(b)$,
	so $j(a\ra b) \in E_{ab}$.
	Now we will show that $j(a\ra b)$ is the maximum of $E_{ab}$.
	In order to see this, let $W\in E_{ab}$, i.e., $W\in D$ and $W\cap j(a) \subseteq j(b)$.
	We will show that $W\subseteq j(a\ra b)$. Let $F\in \IFil(A)$ and suppose that $F\in W$ and $F\notin j(a\ra b)$.
	In particular, $a\ra b \notin F$. Then it follows from Lemma \ref{il3} that there exists $G\in \IFil(A)$
	such that $a\in G$, $b\notin G$ and $F\cap \sq{A} \subseteq G$. Besides, it follows from Lemma \ref{il4}
	that $G\in W$, so $G\in W\cap j(a) \subseteq j(b)$. Thus, $b\in G$, which is a contradiction.
	Hence, $W\subseteq j(a\ra b)$. Therefore, $j(a\ra b)$ is the maximum of $E_{ab}$, which was our aim.
\end{proof}

The following result follows from Proposition \ref{pp}.

\begin{thm} \label{maingoal1}
	Let $A\in \sha$. Then, $(\IFil(A)^+, D)$ with
	$D$ as defined above is a subresiduated lattice where $j(A)$
	is isomorphic to $A$.
\end{thm}

It follows from Theorem \ref{maingoal1} that every sub-Hilbert algebra is a
$\{\ra\}$-subreduct of a subresiduated lattice.

\begin{prop}\label{sHLesIK}
	$\sha \subseteq \mathcal{K}$.
\end{prop}

It follows from propositions \ref{KessHL} and \ref{sHLesIK} that
\[
\mathcal{K} = \sha.
\]
Then, we obtain the following result.

\begin{cor}
	The class $\mathcal{K}$ is a quasivariety.
\end{cor}

\section{On the $\{\wedge, \ra\}$-subreducts of subresiduated lattices}\label{semireticulos}

In this section we study the class whose members are the $\{\wedge, \ra\}$-subreducts
of subresiduated lattices. We show that this class is a variety and we give
an equational basis for it.

\begin{defn}\label{defdesrs}
	An algebra $(A,\we,\ra,1)$ of type (2,2,0) is a \textbf{subresiduated semilattice} if
	the following equations are satisfied\footnote{Since $(A,\we)$ is a semilattice (by equations \textbf{(SL1)}-\textbf{(SL4)}), it is a poset. The order relation appearing in ''equations'' \textbf{(SR2)} and \textbf{(SR3)} is that associated to the semilattice structure, so, this inequalities may be replaced in an obvious way by actual equations.}:
	\begin{enumerate}
		\item[\textbf{(SL1)}] $x \wedge (y \wedge z) = (x \wedge y) \wedge z$,
		\item[\textbf{(SL2)}] $x \wedge y = y \wedge x$,
		\item[\textbf{(SL3)}] $x \wedge x = x$,
		\item[\textbf{(SL4)}] $x \wedge 1 = x$.
		\item[\textbf{(SR1)}] $(x \we y) \ra y = 1$,
		\item[\textbf{(SR2)}] $x \ra y \leq z \ra (x \ra y)$,
		\item[\textbf{(SR3)}] $x \we (x \ra y) \leq y$,
		\item[\textbf{(SR4)}] $z \ra (x \we y) = (z \ra x) \we (z \ra y)$.	
	\end{enumerate}
\end{defn}

In an equivalent way, an algebra $(A,\we,\ra,1)$ of type (2,2,0) is a subresiduated
semilattice if $(A,\we,1)$ is a bounded semilattice (i.e., a semilattice with a greatest element) and the equations
$\textbf{(SR1)}$ to $\textbf{(SR4)}$ are satisfied. We write $\srs$ to indicate the variety whose members
are subresiduated semilattices.
The definition of subresiduated semilattice is motivated by \cite[Theorem 1]{EH}.
\vskip.2cm

Let $A \in \srs$ and $a\in A$. As in the case of subresiduated lattices and sub-Hilbert algebras,
we define $\sq a = 1\ra a$ and $\sq A := \{\sq a \ | \ a \in A\}$.

\begin{lem}\label{lemant}
	Let $A \in \srs$. The following conditions are satisfied for every $a,b,c\in A$:
	\begin{enumerate}
		\item If $a\leq b$, then $c\ra a \leq c\ra b$.
		\item $(a\ra b)\we (b\ra c) \leq a\ra c$.
		\item $a\leq b$ if and only if $a\ra b = 1$.
		\item If $a\leq b$, then $b\ra c \leq a\ra c$.
		\item $\sq(a\ra b) = a\ra b$.
	\end{enumerate}
\end{lem}

\begin{proof}
	Item 1. follows from $\textbf{(SR4)}$.
	
	In order to prove 2. let $a,b,c \in A$. Let us first note the fact that $x \ra x = 1$ holds in $\srs$. That is immediate from $\textbf{(SR1)}$. It follows from $\textbf{(SR3)}$ that $a\we (a\ra b) \leq b$ and
	$b\we (b\ra c)\leq c$. Then
	\[
	a\we (a\ra b) \we (b\ra c) \leq c.
	\]
	Thus, it follows from 1. that
	\begin{equation} \label{eqant1}
		a\ra[a\we (a\ra b) \we (b\ra c)] \leq a\ra c.
	\end{equation}
	But, by $\textbf{(SR1)}$, $\textbf{(SR2)}$ and $\textbf{(SR4)}$,
	\[
	\begin{array}
		[c]{lllll}
		a\ra[a\we (a\ra b) \we (b\ra c)] &     =  &(a\ra a)\we [a\ra (a\ra b)] \we [a\ra (b\ra c)] &  & \\
		&     =  & [a\ra (a\ra b)] \we [a\ra (b\ra c)] &  &\\
		& \geq   & (a\ra b) \we (b\ra c). &  &
	\end{array}
	\]
	Hence,
	\begin{equation}\label{eqant2}
		(a\ra b) \we (b\ra c) \leq a\ra[a\we (a\ra b) \we (b\ra c)].
	\end{equation}
	So, by (\ref{eqant1}) and (\ref{eqant2}),
	$(a\ra b) \we (b\ra c) \leq a\ra c$.
	
	Now we will see 3. Let $a,b\in A$. Suppose that $a\leq b$. Then, by $\textbf{(SR1)}$,
	$a\ra b = (a\we b)\ra b = 1$, so $a\ra b = 1$. Conversely, assume that
	$a\ra b = 1$. Thus it follows from $\textbf{(SR3)}$ that $a = a \we (a\ra b) \leq b$,
	so $a\leq b$.
	
	Now we will show 4. Let $a,b\in A$ such that $a\leq b$.
	Taking into account 2. and 3.,
	$b\ra c =  1 \we (b\ra c) = (a\ra b) \we (b\ra c) \leq a\ra c$.
	
	Finally, 5. follows from $\textbf{(SR2)}$ and $\textbf{(SR3)}$.
\end{proof}

Note that it follows from 5. of Lemma \ref{lemant} that 
\[
\sq A = \{a\in A: \sq a = a\}.
\]

It naturally arise the question if some characterization as pairs can be given for
subresiduated semilattices. This is in fact the case.

\begin{defn}\label{srs-pair}
	An SRS-pair is a pair $(A,D)$ such that $A$ is a bounded semilattice, $D$ is a bounded semilattice of $A$, 
	and for every $a,b\in A$ there exists the maximum of the set $\{d \in \sq A \ | \ d\wedge a \leq b\}$,
	which will be denoted by $a\ra_D b$. If there is not ambiguity we write $a\ra b$ in place of $a\ra_D b$.
\end{defn}

Note that if $(A,D)$ is a SRS-pair then $D = \{a\in A:1\ra a = a\}$
and that SRS-pairs can be seen as algebras $(A,\we,\ra,1)$ of type $(2,2,0)$.   

\begin{thm}\label{lemao}
	An algebra $(A,\we,\ra,1)$ of type $(2,2,0)$ is a subresiduated semilattice if and only if
	it is a SRS-pair.
\end{thm}

\begin{proof}
	Let $A\in \srs$.
	In order to show that $\sq A$ is a bounded semilattice of $A$, first note that $\sq 1 = 1$. 
	Besides, by $\textbf{(SR4)}$, $\sq(a \we b) = \sq a \we \sq b$.
	Thus, $\sq A$ is a bounded semilattice of $A$.
	Now we will see that for every $a,b\in A$ there exists the maximum of the set $E_{a b} : = \{d \in \sq A \ | \ d\we a \leq b\}$.
	Let $a,b\in A$.
	By Lemma \ref{lemant}, $a\ra b\in \sq A$. Besides, by $\textbf{(SR3)}$, $a \we (a \ra b) \leq b$. Then $a \ra b \in E_{a b}$.
	Let now $d \in E_{a b}$, so $d\in \sq A$ and $a \we d \leq b$. It follows from Lemma \ref{lemant}
	that $a \ra (a \we d) \leq a \ra b$. Besides, by $\textbf{(SR1)}$ and $\textbf{(SR4)}$,
	\[
	a \ra (a \we d) = (a \ra a) \we (a \ra d) = a \ra d.
	\]
	Thus,
	\begin{equation}\label{max1}
		a\ra d \leq a\ra b.
	\end{equation}
	On the other hand, since $a\leq 1$, by Lemma \ref{lemant} we get
	\begin{equation} \label{max2}
		1 \ra d \leq a \ra d.
	\end{equation}
	Thus, by (\ref{max1}), (\ref{max2}) and the fact that $d = \sq d$, we deduce that
	$d \leq a \ra b$.
	Therefore, $a \ra b$ is the maximum of $E_{a b}$.
	
	Conversely, let $(A,D)$ be a SRS-pair.
	The condition $\textbf{(SR1)}$ follows from that $1\in D$.
	Also note that for every $a,b,c\in A$, since $a\ra b\in D$ and $(a\ra b)\we c \leq a\ra b$ then $a\ra b \leq c\ra (a\ra b)$,
	which is $\textbf{(SR2)}$. The condition $\textbf{(SR3)}$ is immediate. In order to show $\textbf{(SR4)}$,
	first we will see that if $a,b,c\in A$ and $a\leq b$ then $c\ra a \leq c\ra b$. Suppose that $a\leq b$.
	Then $c\we (c\ra a) \leq a\leq b$, so $c\we (c\ra a)\leq b$. Taking into account that $c\ra a\in D$
	we get $c\ra a\leq c\ra b$. Now consider $a,b,c$ arbitrary elements of $A$. Since $a\we b\leq a$ and $a\we b\leq b$ then
	$c\ra (a\we b)\leq c\ra a$ and $c\ra (a\we b)\leq c\ra a$, so $c\ra (a\we b)$ is a lower bound of $\{c\ra a,c\ra b\}$ in $D$.
	Let $d$ be a lower bound of $\{c\ra a,c\ra b\}$ in $D$. Then $d\in D$, $d\leq c\ra a$ and $d\leq c\ra b$, so
	$d\we c\leq a$ and $d\we c\leq b$. Thus, $d\we c\leq a\we b$, so $d\leq c\ra (a\we b)$. Thus,
	$c\ra (a \we b)$ is the infimum of $\{c\ra a,c\ra b\}$ in $D$. But since $D$ is a semilattice of $A$
	then the infimum of $\{c\ra a,c\ra b\}$ in $D$ is equal to the infimum of $\{c\ra a,c\ra b\}$ in $A$.
	Therefore, $c\ra (a\we b) = (c\ra a)\we (c\ra b)$.
\end{proof}

Let $A \in \srs$. We write $\Fil(A)$ to indicate the set of filters of $A$. It follows from $\textbf{(SR3)}$
that $\Fil(A)\subseteq \IFil(A)$ (the converse inclusion is not true in general).
Moreover, $\Fil(A)^+$ is a complete lattice (which is bounded and distributive).
Let $j : A \to \Fil(A)^+$ be the map given by $j(a) := \{F \in \Fil(A)\ | \ a \in F\}$. It is immediate
that $j$ is injective, $j(1) = \Fil(A)$ and $j(a\we b) = j(a) \cap j(b)$ for every $a,b\in A$.

Then we have the following result.

\begin{lem}
	Let $A\in \srs$. Then $j$ is an embedding of bounded semilattices.
\end{lem}

Let $A\in \srs$. We define $T: = j(\sq A)$, which is a subset of $\Fil(A)^+$. We also define $D$
as the complete sublattice of $\Fil(A)^+$ generated by $T$.
Following a reasoning similar to the employed in Section \ref{subHilbert}
it is possible to show that the pair $(\Fil(A)^+,D)$ is a subresiduated lattice.
We write $\Ra$ for the implication in this algebra, i.e., for every $U,V\in \Fil(A)^+$,
$U\Ra V: = \max\{W\in D: W\cap U\subseteq V\}$.

Our next aim is to show that
$j(a\ra b) = j(a) \Ra j(b)$. In order to prove it we will see the following lemma.

\begin{lem} \label{is3}
	Let $A\in \srs$, $F \in \Fil(A)$ and $a, b \in A$ such that $a \ra b \notin F$.
	Then there exists $G \in \Fil(A)$ such that $a \in G$, $b \notin G$ and $F \cap \sq A \subseteq G$.
\end{lem}

\begin{proof}
	Let $A\in \srs$, $F \in \Fil(A)$ and $a, b \in A$ such that $a \ra b \notin F$.
	We define $G$ as the filter generated by $(F \cap \sq A) \cup \{a\}$.
	Then $a \in G$ and $F \cap \sq A \subseteq G$.
	
	Suppose that $b \in G$. Then there exists $x \in (F \cap \sq A)$ such that $x \we a \leq b$.
	Thus, it follows from Lemma \ref{lemant}, SR1. and SR4. that
	\[
	\begin{array}
		[c]{lllll}
		a\ra b                                  &  \geq  &a \ra (a \we x) &  & \\
		&     =  & (a \ra a) \we (a \ra x) &  &\\
		&     =  & a\ra x, &  &
	\end{array}
	\]
	so
	\begin{equation}\label{eqf}
		a\ra x \leq a\ra b.
	\end{equation}
	Besides, it follows again
	by Lemma \ref{lemant} and the inequality $a\leq 1$ that
	\begin{equation}\label{eqt}
		1 \ra x \leq a \ra x.
	\end{equation}
	Since $x \in \sq A$ then it follows from (\ref{eqf}) and (\ref{eqt}) that
	$x \leq a \ra b$. But $x \in F$, so $a \ra b \in F$, which is a contradiction.
	Therefore, $b \notin G$. 	
\end{proof}

\begin{lem} \label{is4}
	Let $A\in \srs$, $F, G \in \Fil(A)$ and $W \in D$ with $F \in W$
	and $F \cap \sq A \subseteq G$. Then $G \in W$.
\end{lem}

\begin{proof}
	Analogous to the proof Lemma \ref{il4}.
\end{proof}

\begin{prop}
	Let $A\in \srs$ and $a, b \in A$. Then
	\[
	j(a \ra b) = j(a) \Ra j(b).
	\]
\end{prop}

\begin{proof}
	Let $A\in \srs$ and $a,b\in A$. We define the set
	\[
	E_{ab} := \{W\in D: W\cap j(a) \subseteq j(b)\}.
	\]
	It is immediate that $j(a\ra b)\cap j(a) \subseteq j(b)$,
	so $j(a\ra b) \in E_{ab}$.
	Now we will show that $j(a\ra b)$ is the maximum of $E_{ab}$.
	In order to see this, let $W\in E_{ab}$, i.e., $W\in D$ and $W\cap j(a) \subseteq j(b)$.
	We will show that $W\subseteq j(a\ra b)$. Let $F\in \Fil(A)$ and suppose that $F\in W$ and $F\notin j(a\ra b)$.
	In particular, $a\ra b \notin F$. Then it follows from Lemma \ref{is3} that there exists $G\in \Fil(A)$
	such that $a\in G$, $b\notin G$ and $F\cap \sq{A} \subseteq G$. Besides, it follows from Lemma \ref{is4}
	that $G\in W$, so $G\in W\cap j(a) \subseteq j(b)$. Thus, $b\in G$, which is a contradiction.
	Hence, $W\subseteq j(a\ra b)$. Therefore, $j(a\ra b)$ is the maximum of $E_{ab}$, which was our aim.
\end{proof}

\begin{cor} \label{corSRS}
	Let $A \in \srs$. Then $j$ is an embedding from $A$ to the $\{\ra, \we\}$-reduct of $(\Fil(A)^+, D)$.
\end{cor}

On the other hand, it is immediate that every $\{\ra, \wedge\}$-subreducts of subresiduated lattices is an element of $\srs$ \cite[Theorem 1]{EH}. The following result follows from this and Corollary \ref{corSRS}.

\begin{thm}
	The variety $\srs$ is the class of $\{\ra, \wedge\}$-subreducts of subresiduated lattices.
\end{thm}

\section{A bit of logic}
\label{logic}

The aim of this section is to propose an algebraizable propositional logic whose algebraic semantics is the class of subresiduated lattices. We also explore several reducts and variations of this logic and their algebraic semantics. 

Along this section, in order to simplify the notation, whenever the logic is clear from the context, we shall just write $\vdash$ for its entailment relation. When some confusion can arise, we shall indicate the logic with a suitable subindex.
\vskip.2cm
Let us begin by recalling the system $R4$ presented by a Hilbert style system in the language $\mathcal{L} = \{\ra, \vee, \wedge, \neg \}$, as it is done in \cite{EH}.

\subsection{The logic $R4$}
\label{R4}
In \cite{EH}, a Hilbert style system for Lewy calculus is proposed by Epstein and Horn. There, it is shown that the Lindembaun-Tarski algebra of this calculus is a subresiduated lattice. Hence, a motivation for defining subresiduated lattices is as an algebraic semantics for this calculus. In order to simplify future comparisons with another calculus introduced in this article, we summarize in this subsection the most relevant results of \cite{EH} concerning Lewy calculus.

\begin{defn}\label{logicR4}
	The $R4$ calculus is the calculus, in the intuitionistic language $\mathcal{L} = \{\ra, \vee, \wedge, \neg \}$, which can be presented by the following Hilbert style system.
	\vskip.2cm
	\noindent\textbf{Axiom schemes:}
	\vskip.2cm
	\begin{enumerate}
		\item[\textbf{(A1)}] $\alpha \to \alpha$,
		\item[\textbf{(A2)}] $(\alpha \ra \beta) \ra (\delta \ra (\alpha \ra \beta))$,
		\item[\textbf{(A3)}] $(\alpha \ra (\beta \ra \delta)) \ra ((\alpha \ra \beta) \ra (\alpha \ra \delta))$,
		\item[\textbf{(A4)}] $(\alpha \wedge \beta) \ra \alpha$,
		\item[\textbf{(A5)}] $(\alpha \wedge \beta) \ra \beta$,
		\item[\textbf{(A6)}] $(\delta \ra \alpha) \ra ((\delta \ra \beta) \ra (\delta \ra (\alpha \wedge \beta))$,
		\item[\textbf{(A7)}] $\alpha \ra (\alpha \vee \beta)$,
		\item[\textbf{(A8)}] $\beta \ra (\alpha \vee \beta)$,
		\item[\textbf{(A9)}] $(\alpha \ra \delta) \ra ((\beta \ra \delta) \ra ((\alpha \vee \beta) \ra \delta)$,
		\item[\textbf{(A10)}] $(\alpha \wedge (\beta \vee \delta)) \ra ((\alpha \wedge \beta) \vee (\alpha \wedge \delta))$,
		\item[\textbf{(A11)}] $\neg \alpha \ra (\alpha \ra \beta)$, \ \ and
		\item[\textbf{(A12)}] $(\alpha \ra \neg \alpha) \ra \neg \alpha$.
	\end{enumerate}
	\vskip.3cm
	
	\noindent\textbf{Rules:}
	\vskip.2cm
	\[
	\dfrac{\alpha, \ \alpha \ra \beta}{\beta}\ \ \mathbf{(MP)} 
	\]
	\vskip.2cm
\end{defn}

Write IR4 for the implicative fragment of R4; i.e., the calculus satisfying the axiom schemes \textbf{(A1)}-\textbf{(A3)} and $\mathbf{(MP)}$. The following weak version of the deduction theorem for $R4$ ($IR4$) noteworthy simplifies proofs.

\begin{lem}[\cite{EH}, page 202.]
	\label{DTdeIR4}
	Let $\alpha_1, \cdots, \alpha_n \vdash \beta$ and for $i = 1, \cdots, n-1$, $\alpha_i = \delta_i \ra \eta_i$ for some formulae $\delta_i, \eta_i$, then $\alpha_1, \cdots, \alpha_{n-1} \vdash \alpha_n \ra \beta$. In particular, if $\alpha \vdash \beta$, then $\vdash \alpha \ra \beta$.
\end{lem}
\vskip.2cm

As an application of Lemma \ref{DTdeIR4}, we can give the following simple proof of the fact that the following is a valid scheme of $R4$ (in fact, of $IR4$):
$$\vdash (\alpha \ra \beta) \ra ((\beta \ra \delta) \ra (\alpha \ra \delta))$$
\vskip.2cm
Formally,

\begin{tabular}{lll}
	1. & $\alpha$ & by hypothesis \\
	2. & $\alpha \ra \beta$ & by hypothesis \\
	3. & $\beta \ra \delta$ & by hypothesis \\
	4. & $\beta$ & from 1. and 2. by \textbf{(MP)}\\
	5. & $\delta$ & from 3. and 4. by \textbf{(MP)}\\
\end{tabular}
\vskip.2cm
So, we have that $\{\alpha \ra \beta, \beta \ra \delta, \alpha\} \vdash \delta$. Applying Lemma \ref{DTdeIR4}, we get $\{\alpha \ra \beta, \beta \ra \delta\} \vdash \alpha \ra \delta$.

Finally, applying another two times Lemma \ref{DTdeIR4}, we get
$$\vdash (\alpha \ra \beta) \ra ((\beta \ra \delta) \ra (\alpha \ra \delta))$$

\vskip.2cm
Furthermore, since every axiom of $R4$ ($IR4$) is a hypothetical formula, the following derived rule holds in $R4$ ($IR4$):
\vskip.2cm
\noindent\textbf{Derived rule $\mathbf{(wT)}$:}
\vskip.2cm
\[
\dfrac{\alpha}{\beta \ra \alpha}, \textrm{ whenever $\vdash \alpha$.}
\]

\subsection{The logic $R4^\star$}
\label{subHilbertLogic}

Following the notation introduced when we were tackling sub-Hilbert algebras, we write
$\sq \alpha$ as a shorthand for $(\alpha \ra \alpha) \ra \alpha$, for any formula $\alpha$ in $\mathcal{L}$.
\vskip.3cm

\begin{defn}\label{logicR4s}
	The $R4^\star$ calculus is the calculus, in the language of $R4$, which can be presented by the following Hilbert style system.
\vskip.2cm

\noindent\textbf{Axiom schemes:}
\vskip.2cm
\begin{enumerate}
	\item[\textbf{(Ax1)}] $\alpha \to \alpha$,
	\item[\textbf{(Ax2)}] $(\alpha \ra \beta) \ra ((\beta \ra \delta) \ra (\alpha \ra \delta))$,
	\item[\textbf{(Ax3)}] $(\alpha \ra (\beta \ra \delta)) \ra ((\alpha \ra \beta) \ra (\alpha \ra \delta))$,
	\item[\textbf{(C1)}] $(\alpha \wedge \beta) \ra \alpha$,
	\item[\textbf{(C2)}] $(\alpha \wedge \beta) \ra \beta$,
	\item[\textbf{(C3)}] $(\delta \ra \alpha) \ra ((\delta \ra \beta) \ra (\delta \ra (\alpha \wedge \beta))$,
	\item[\textbf{(D1)}] $\alpha \ra (\alpha \vee \beta)$,
	\item[\textbf{(D2)}] $\beta \ra (\alpha \vee \beta)$,
	\item[\textbf{(D3)}] $(\alpha \ra \delta) \ra ((\beta \ra \delta) \ra ((\alpha \vee \beta) \ra \delta)$,
	\item[\textbf{(N1)}] $\neg \alpha \ra (\alpha \ra \beta)$, 
	\item[\textbf{(N2)}]  $(\alpha \ra \neg \alpha) \ra \neg \alpha$ \ \ and
	\item[\textbf{(Dist)}] $(\alpha \wedge (\beta \vee \delta)) \ra ((\alpha \wedge \beta) \vee (\alpha \wedge \delta))$.
\end{enumerate}
\vskip.3cm

\noindent\textbf{Rules:}
\vskip.2cm
\[
\dfrac{\alpha, \ \alpha \ra \beta}{\beta}\ \ \mathbf{(MP)} \ \ \ \ \ \ \ \ \ \ \ \ \ \ \ \ \ \ \ \ \dfrac{\alpha}{\beta \ra \alpha} \ \ \mathbf{(T)}
\]
\end{defn}
\vskip.2cm

Note that systems $R4$ and $R4^\star$ share all axiom schemes except \textbf{(A2)} and that systems $R4$ have a weaker version of rule $\mathbf{(T)}$.

Now, let us state some properties of this logic which will ease its comparison with the $R4$ calculus.

\begin{lem}\label{reglasderivadas}
Let $\alpha$, $\beta$ and $\delta$ be  arbitrary formulas in the language  $\mathcal{L}$. The following are derived rules for $R4^\star$:
\begin{enumerate}
	\item $\alpha \ra \beta, \beta \ra \delta \vdash \alpha \ra \delta$,
	\item $\vdash \alpha  \ra (\beta \ra \beta)$,
	\item $\vdash  ((\beta \ra \beta) \ra \delta) \ra (\alpha \ra \delta) $ \ \ and
	\item $\vdash (\alpha \ra \beta) \ra (\delta \ra (\alpha \ra \beta))$.
\end{enumerate}
\end{lem}
\begin{proof}
\item[1.] Assume that $\alpha \ra \beta$ and $\beta \ra \delta$. The following is a proof in $R4^\star$ of $\alpha \ra \delta$.
\vskip.2cm
\begin{tabular}{lll}
	1. & $\alpha \ra \beta$ & by hypothesis \\
	2. & $\beta \ra \delta$ & by hypothesis \\
	3. & $(\alpha \ra \beta) \ra ((\beta \ra \delta) \ra (\alpha \ra \delta))$ & by $\textbf{(Ax2)}$ \\
	4. & $(\beta \ra \delta) \ra (\alpha \ra \delta)$ & by $\textbf{(MP)}$ from 1. and 3.\\
	5. & $\alpha \ra \delta$ & by $\textbf{(MP)}$ from 2. and 4.\\
\end{tabular}
\vskip.2cm
Hence, we can conclude that $$\{\alpha \ra \beta, \beta \ra \delta\} \vdash \alpha \ra \delta$$
\vskip.2cm
\item[2.] It is an immediate consequence of $\textbf{(Ax1)}$ and rule $\textbf{(T)}$. Furthermore, we get that  $\alpha \ra \alpha \dashv \vdash \beta \ra \beta$, for every formulae $\alpha$ and $\beta$.
\vskip.2cm
\item[3.] Let $\alpha$, $\beta$ and $\delta$ be formulae, then,
\vskip.2cm
\begin{tabular}{lll}
	1. & $(\alpha \ra (\beta \ra \beta)) \ra (((\beta \ra \beta) \ra \delta) \ra (\alpha \ra \delta))$ & by \textbf{(Ax2)} \\
	2. & $\alpha \ra (\beta \ra \beta)$ & by item 2 \\
	3. & $((\beta \ra \beta) \ra \delta) \ra (\alpha \ra \delta)$ & by \textbf{(MP)} from 1. and 2. \\	
\end{tabular}
\vskip.2cm
\item[4.] Let $\alpha$, $\beta$ and $\delta$ be formulae, then,
\vskip.2cm
\begin{tabular}{lll}
	1. & $(\alpha \ra \beta) \ra ((\beta \ra \beta) \ra (\alpha \ra \beta))$ & by \textbf{(Ax2)} \\
	2. & $((\beta \ra \beta) \ra (\alpha \ra \beta)) \ra (\delta \ra (\alpha \ra \beta))$ & by item 3.\\
	3. & $(\alpha \ra \beta) \ra (\delta \ra (\alpha \ra \beta))$ & from 1. and 2. by item 1\\
\end{tabular}
\end{proof}

Note that the scheme $(\alpha \ra \beta) \ra (\delta \ra (\alpha \ra \beta))$
of item 4. of previous lemma is axiom scheme \textbf{(A2)} of logic $R4$. As a consequence, we get that every deduction in $R4$ is valid in $R4^\star$.
\vskip.2cm

On the other hand, since $\textbf{(Ax2)}$ is a theorem in the implicative fragment of R4\footnote{Equation (17) in page 202 of \cite{EH}; see also the paragraph after Lemma \ref{DTdeIR4}.}, every deduction in $R4^\star$ not involving the use of rule $\textbf{(T)}$ is valid in $R4$. Note that both logics share the same theorems.

So, as an immediate consequence of Lemma \ref{DTdeIR4} we get the following result.

\begin{cor}\label{TDparaL_1}
	Let $\alpha_1, \cdots, \alpha_n \vdash \beta$ be an entailment of $\beta$ whose proof  doesn't require the use of rule $\textbf{(T)}$ and suppose that for $i = 1, \cdots, n-1$, $\alpha_i = \delta_i \ra \eta_i$ for some formulae $\delta_i, \eta_i$, then $\alpha_1, \cdots, \alpha_{n-1} \vdash \alpha_n \ra \beta$. In particular, if $\alpha \vdash \beta$, and the proof of $\beta$ doesn't require the use of rule $\textbf{(T)}$, then $\vdash \alpha \ra \beta$.
\end{cor}

\subsection{$R4^\star$ is an implicative logic}
\label{implicativelogic}
\vskip.2cm
Let us now see that system $R4^\star$ is an implicative logic in the sense of Rasiowa\footnote{See Chapter 2 of \cite{F} for the definition and properties of implicative logics.}. 
\vskip.2cm

We start by seeing that the implicative fragment of $R4^\star$, which we shall refer as $IR4^\star$, is an implicative logic. Then we shall check the compatibility of the other connectives with respect to the natural congruence.
\vskip.2cm 
\noindent \textbf{(IL1)} By $\textbf{(Ax1)}$, $\vdash \alpha \ra \alpha$.
\vskip.2cm

\noindent \textbf{(IL2)} It is rule 1. of Lemma \ref{reglasderivadas}.
\vskip.2cm
\noindent \textbf{(IL3)} Recall that condition \textbf{(IL3)} ask for
\begin{equation*}
	\{\alpha \ra \beta, \beta \ra \alpha, \delta \ra \eta, \eta \ra \delta\} \vdash (\alpha \ra \delta) \ra (\beta \ra \eta).
\end{equation*}

In order to prove it we shall first see the validity of the following deductions in $IR4^\star$.
\vskip.2cm

\begin{tabular}{lll}
	\textbf{(IL3$_1$)} & $\{\alpha \ra \beta, \beta \ra \alpha\} \vdash (\alpha \ra \delta) \ra (\beta \ra \delta)$\\
	\textbf{(IL3$_2$)} & $\{\alpha \ra \beta, \beta \ra \alpha\} \vdash (\delta \ra \alpha) \ra (\delta \ra \beta)$\\
\end{tabular}
\vskip.2cm
Note that the roles of $\alpha$ and of $\beta$ are interchangeable in previous expressions. Hence, we have the following result.
\begin{lem}
	Assume that $IR4^\star$ satisfies \textbf{(IL3$_1$)} and \textbf{(IL3$_2$)}, then it satisfies \textbf{(IL3)}.
\end{lem}
\begin{proof}
	Assume $\alpha \ra \beta$, $\beta \ra \alpha$, $\delta \ra \eta$ and $\eta \ra \delta$. By \textbf{(IL3$_1$)},
	\begin{equation}\label{il31}
		\{\alpha \ra \beta, \beta \ra \alpha\} \vdash (\alpha \ra \delta) \ra (\beta \ra \delta).
	\end{equation}
	By \textbf{(IL3$_2$)},
	\begin{equation}\label{il32}
		\{\delta \ra \eta, \eta \ra \delta\} \vdash (\beta \ra \delta) \ra (\beta \ra \eta).
	\end{equation}
	Hence, the following is a proof in $IR4^\star$:
	
	\begin{tabular}{lll}
		1. & $\alpha \ra \beta$ & by hypothesis \\
		2. & $\beta \ra \alpha$ & by hypothesis \\
		3. & $\delta \ra \eta$ & by hypothesis \\
		4. & $\eta \ra \delta$ & by hypothesis \\
		5. & $(\alpha \ra \delta) \ra (\beta \ra \delta)$ & from 1. and 2., applying \eqref{il31} \\
		6. & $(\beta \ra \delta) \ra (\beta \ra \eta)$ & from 3. and 4., applying \eqref{il32}\\
		7. & $(\alpha \ra \delta) \ra (\beta \ra \eta)$ & by \textbf{(IL2)}, from 5. and 6.\\
	\end{tabular}
	\vskip.2cm
	
	So, we conclude that
	\begin{equation*}\label{il3'}
		\{\alpha \ra \beta, \beta \ra \alpha, \delta \ra \eta, \eta \ra \delta\} \vdash (\alpha \ra \delta) \ra (\beta \ra \eta).
	\end{equation*}
\end{proof}

Hence, it suffices to give proofs of \textbf{(IL3$_1$)} and \textbf{(IL3$_2$)} in $IR4^\star$.
\vskip.2cm
\noindent Let us start with \textbf{(IL3$_1$)}:\vskip.2cm

\begin{tabular}{lll}
	1. & $\alpha \ra \beta$ & by hypothesis \\
	2. & $\beta \ra \alpha$ & by hypothesis \\
	3. & $(\beta \ra \alpha) \ra ((\alpha \ra \delta) \ra (\beta \ra \delta))$ & \textbf{(Ax2)}\\
	4. & $(\alpha \ra \delta) \ra (\beta \ra \delta)$ & by \textbf{(MP)} from 2. and 3.
\end{tabular}
\vskip.2cm

Hence, it holds \textbf{(IL3$_1$)}.
\vskip.2cm

\noindent Let us now prove \textbf{(IL3$_2$)}:\vskip.2cm

\begin{tabular}{lll}
	1. & $\delta \ra \eta$ & by hypothesis \\
	2. & $\eta \ra \delta$ & by hypothesis \\
	3. & $\beta \ra (\delta \ra \eta)$ & from 1., by rule \textbf{(T)}\\
	4. & $(\beta \ra (\delta \ra \eta)) \ra ((\beta \ra \delta) \ra (\beta \ra \eta))$ & \textbf{(Ax3)}\\
	7. & $(\beta \ra \delta) \ra (\beta \ra \eta)$ & by \textbf{(MP)} from 3. and 4.\\
\end{tabular}
\vskip.2cm

Then, \textbf{(IL3$_2$)} also holds.

\vskip.2cm
\noindent \textbf{(IL4)} Is Modus Ponens.
\vskip.2cm

\vskip.2cm
\noindent \textbf{(IL5)} Is rule \textbf{(T)}.
\vskip.2cm

Since system $IR4^\star$ satisfies properties \textbf{(IL1)} to \textbf{(IL5)}, it is an implicative logic.
\vskip.2cm

Let us now check that the expansion of $IR4^\star$ with the other connectives remains an implicative logic. We have to check that the following properties are provable in $R4^\star$

\begin{enumerate}
	\item[\textbf{(IL3$_\wedge$)}] $\{\alpha \ra \beta, \beta \ra \alpha\} \vdash (\alpha \wedge \delta) \ra (\beta \wedge \delta)$.
	\item[\textbf{(IL3$_\neg$)}] $\{\alpha \ra \beta, \beta \ra \alpha\} \vdash \neg \alpha  \ra \neg \beta$.
	\item[\textbf{(IL3$_\vee$})] $\{\alpha \ra \beta, \beta \ra \alpha\} \vdash (\alpha \vee \delta) \ra (\beta \vee \delta)$.
\end{enumerate}
\vskip.2cm
Let us start giving a proof of \textbf{(IL3$_\wedge$)}. Let $\alpha$, $\beta$ and $\delta$ be formulas of $R4^\star$ and assume that $\alpha \ra \beta$, $\beta \ra \alpha$. Then,

\vskip.2cm
\begin{tabular}{lll}
	1. & $\alpha \ra \beta$ & by hypothesis \\
	2. & $\beta \ra \alpha$ & by hypothesis \\
	3. & $(\alpha \wedge \delta) \ra \alpha$ & \textbf{(C1)}\\
	4. & $(\alpha \wedge \delta) \ra \beta$ & from 1. and 3. by \textbf{(IL2)}\\
	5. & $(\alpha \wedge \delta) \ra \delta$ & \textbf{(C2)}\\
	6. & $((\alpha \wedge \delta) \ra \beta) \ra \left(((\alpha \wedge \delta) \ra \delta) \ra ((\alpha \wedge \delta) \ra (\beta \wedge \delta)) \right) $ & \textbf{(C3)}\\
	7. & $((\alpha \wedge \delta) \ra \delta) \ra ((\alpha \wedge \delta) \ra (\beta \wedge \delta))$ & by \textbf{(MP)}, from 4. and 6.\\
	8. & $(\alpha \wedge \delta) \ra (\beta \wedge \delta)$ & by \textbf{(MP)}, from 5. and 7.\\
\end{tabular}
\vskip.2cm

Let us now check \textbf{(IL3$_\neg$)}. Let $\alpha$ and $\beta$ be formulas in $R4^\star$ and assume that $\alpha \ra \beta$, $\beta \ra \alpha$. The following is a proof in $R4^\star$.
\vskip.2cm
\begin{tabular}{lll}
	1. & $\beta \ra \alpha$ & by hypothesis \\
	2. & $(\beta \ra \alpha) \ra \left((\alpha \ra \neg \beta) \ra (\beta \ra \neg \beta))\right) $ & \textbf{(Ax2)} \\
	3. & $(\alpha \ra \neg \beta) \ra (\beta \ra \neg \beta)$ & by \textbf{(MP)}, from 1. and 2.\\
	4. & $\neg \alpha \ra (\alpha \ra \neg \beta)$ & \textbf{(N1)}\\
	5. & $\neg \alpha \ra (\beta \ra \neg \beta)$ & from 4. and 3. by \textbf{(IL2)}\\
	6. & $(\beta \ra \neg \beta) \ra \neg \beta$ & \textbf{(N2)}\\
	7. & $\neg \alpha \ra \neg \beta$ & from 5. and 6. by \textbf{(IL2)}\\
\end{tabular}
\vskip.2cm

Let us finally check \textbf{(IL3$_\vee$)}. For $\alpha$, $\beta$ and $\delta$ formulas in $R4^\star$, assume that $\alpha \ra \beta$ and $\beta \ra \alpha$. The following is a proof in $R4^\star$.
\vskip.2cm
\begin{tabular}{lll}
	1. & $\beta \ra \alpha$ & by hypothesis \\
	2. & $\alpha \ra (\alpha \vee \delta)$ & \textbf{(D1)}\\
	3. & $\beta \ra (\alpha \vee \delta)$ & from 1. and 2. by \textbf{(IL2)}\\
	4. & $\delta \ra (\alpha \vee \delta)$ & \textbf{(D2)}\\
	5. & $(\beta \ra (\alpha \vee \delta)) \ra ((\delta \ra (\alpha \vee \delta)) \ra ((\beta \vee \delta) \ra (\alpha \vee \delta)))$ & \textbf{(D3)}\\
	6. & $(\delta \ra (\alpha \vee \delta)) \ra ((\beta \vee \delta) \ra (\alpha \vee \delta))$  & by \textbf{(MP)} from 3. and 5.\\
	7. & $(\beta \vee \delta) \ra (\alpha \vee \delta)$ & by \textbf{(MP)} from 4. and 6.\\
\end{tabular}
\vskip.2cm
Note that in each of the proofs of \textbf{(IL3$_\wedge$)}, \textbf{(IL3$_\neg$)} and \textbf{(IL3$_\vee$)} only the axioms \textbf{(Ax1)}-\textbf{(Ax3)}, rules \textbf{(MP)} and \textbf{(T)} and the axioms for corresponding connective are used. Hence, each of the fragments of $R4^\star$ containing $\ra$ and any other combination of the remaining connectives (including none) will still be implicative.
\vskip.2cm

\subsection{An algebraic semantics for $R4^\star$}
\label{algebraicsemantics}
\vskip.2cm

In previous subsection we have shown that $R4^\star$ is an algebraic logic in the sense of \cite{F}. Hence it has an equivalent algebraic semantics in the sense of Block and Pigozzi. The following metaproperty follows from \cite[Theorem 2.9]{F}.

\begin{thm}[Completeness Theorem]\label{completness}
	Logic $R4^\star$ is complete with respect to the class $Alg_{R4^\star}^\ast$ defined below.
\end{thm}
Note that $Alg_{R4^\star}^\ast$ is the class of algebras of type (2,2,2,1,0) in the language $\{ \ra, \wedge, \vee, \neg, 1\}$ defined by the following quasi-identities:
\begin{enumerate}
	\item[\textbf{(I)}] $x \to x = 1$,
	\item[\textbf{(B)}] $(x \ra y) \ra ((y \ra z) \ra (x \ra z)) = 1$,
	\item[\textbf{(S)}] $(x \ra (y \ra z)) \ra ((x \ra y) \ra (x \ra z)) = 1$,
	\item[\textbf{(T)}] $x \to 1 = 1$,
	\item[\textbf{(A)}] if $x \to y = 1$ and $y \to x = 1$ then $x = y$,	
	\item[\textbf{(EC1)}] $(x \wedge y) \ra x = 1$,
	\item[\textbf{(EC2)}] $(x \wedge y) \ra y = 1$,
	\item[\textbf{(EC3)}] $(z \ra x) \ra ((z \ra y) \ra (z \ra (x \wedge y)) = 1$,
	\item[\textbf{(ED1)}] $x \ra (x \vee y) = 1$,
	\item[\textbf{(ED2)}] $y \ra (x \vee y) = 1$, 
	\item[\textbf{(ED3)}] $(x \ra z) \ra ((y \ra z) \ra ((x \vee y) \ra z) = 1$,
	\item[\textbf{(N1)}] $\neg x \ra (x \ra y) = 1$,
	\item[\textbf{(N2)}] $(x \ra \neg x) \ra \neg x = 1$ \ \ and
	\item[\textbf{(Dist)}] $(x \wedge (y \vee z)) \ra ((x \wedge y) \vee (x \wedge z)) = 1$.
\end{enumerate}

\begin{rem}
 Note that quasi-identity \textbf{(A)} is interchangeble with 
 \begin{enumerate}
  \item[\textbf{(MP)}] if $1 \to x = 1$ then $x = 1$.	
 \end{enumerate}
\end{rem}

Let us first note that the class of the implicative reducts of these algebras coincides with the quasivariety $\sha$ of sub-Hilbert Algebras, defined in Section \ref{subHilbert}.
\vskip.3cm
\begin{lem} \label{eqsha}
$Alg_{IR4^\star}^\ast = \sha$.
\end{lem}

\begin{proof}
It follows from \cite[Proposition 2.7]{F} that $Alg_{IR4^\star}^\ast$ is the quasivariety
whose algebras satisfy $\mathbf{(I)}$, $\mathbf{(B)}$, $\mathbf{(S)}$ and $\mathbf{(T)}$
of Definition \ref{sha} with the following quasi-equation: if $1\ra x = 1$ then $x=1$. Moreover,
it follows from \cite[Proposition 2.15]{F} that every algebra of $Alg_{IR4^\star}^\ast$ is an implicative
algebra, so every algebra of $Alg_{IR4^\star}^\ast$ satisfies $\mathbf{(A)}$ of Definition \ref{sha}. Hence, $Alg_{IR4^\star}^\ast \subseteq \sha$.
Conversely, let $A\in \sha$. In order to show that $A\in Alg_{IR4^\star}^\ast$, let $x\in A$ such that $1\ra x = 1$.
Since by $\mathbf{(T)}$, $x\ra 1 = 1$, then it follows from $\mathbf{(A)}$ that $x = 1$. Thus, $A\in \sha$
and in consequence $\sha \subseteq  Alg_{IR4^\star}^\ast$. Therefore, $Alg_{IR4^\star}^\ast = \sha$.
\end{proof}

As an immediate consequence of Theorem \ref{completness} and Lemma \ref{eqsha} we get the following result.

\begin{cor}
	For every $\Gamma \cup \{\alpha\} \subseteq \mathrm{Fml_{L_1}}$, 	
	$\Gamma \vdash \alpha$ if and only if for every $A \in \sha$ and every $h \in \mathrm{Hom}(\mathrm{Fml_{L_1}},A)$, $h(\Gamma) \subseteq \{1\}$ implies $h(\alpha) = 1$.
\end{cor}

Note that the following scheme, which we shall name \textbf{(C4)}, also holds in $R4^\star$.
\begin{equation*}\label{C4}
	\vdash (\alpha \wedge (\alpha \ra \beta)) \ra \beta
\end{equation*}
In order to prove this, we give a proof of $(\alpha \wedge (\alpha \ra \beta)) \vdash \beta$ in $R4^\star$ and apply Corollary \ref{TDparaL_1}:
\vskip.2cm
\begin{tabular}{lll}
	1. & $\alpha \wedge (\alpha \ra \beta)$ & by hypothesis\\
	2. & $(\alpha \wedge (\alpha \ra \beta)) \ra \alpha$ & by \textbf{(C1)} \\
	3. & $(\alpha \wedge (\alpha \ra \beta)) \ra (\alpha \ra \beta)$ & by \textbf{(C2)} \\
	4. & $\alpha \ra \beta$ & from 1. and 3., by \textbf{(MP)}\\
	5. & $\alpha$ & from 1. and 2., by \textbf{(MP)}\\
	6. & $\beta$ & from 5. and 4., by \textbf{(MP)}\\
\end{tabular}
\vskip.2cm

Since \textbf{(C4)} holds in $R4^\star$ and $Alg_{R4^\star}^\ast$ is the algebraic semantics of this logic, we have that the following equation holds in  $Alg_{R4^\star}^\ast$. We label it in order to simplify its future reference.
\begin{enumerate}
	\item[\textbf{(EC4)}] $(x \wedge (x \ra y)) \ra y = 1$.
\end{enumerate}

\begin{lem}\label{SRSessemiret}
	Let $A \in Alg_{R4^\star}^\ast$. Then $(A, \we, 1)$ is a bounded semilattice.
\end{lem}
\begin{proof}
	Let $\leq$ be the natural order of $A$, given by implication. By \textbf{(EC1)}, $x \wedge y \leq x$ and by \textbf{(EC2)}, $x \wedge y \leq y$.
	
	Assume that $z \leq x, y$. Then, $z \ra x = z \ra y = 1$. By \textbf{(EC3)}, we have that $1 \ra (1 \ra (z \ra (x \wedge y))) = 1$. However, by \textbf{(EC4)}, $1 \ra (1 \ra (z \ra (x \wedge y))) \leq z \ra (x \wedge y)$ and, in consequence, $z \leq x \wedge y$; i.e., $x \wedge y = \inf\{x, y\}$.
	
	Furthermore, by \textbf{(T)}, we have that for every $x \in A$, $x \leq 1$.
\end{proof}
Hence, we have the following result.

\begin{lem}\label{baseparaSRS}
	Variety $\srs$ is the equivalent algebraic semantics of the $\{\ra, \wedge\}$-fragment of the logic $R4^\star$.
\end{lem}
\begin{proof}
	By Lemma \ref{SRSessemiret}, every element of $Alg_{R4^\star}^\ast$ is a bounded semilattice and, hence, \textbf{(SL1)} to \textbf{(SL4)} are satisfied.
	
	On the other hand, \textbf{(SR1)} is \textbf{(EC1)}, \textbf{(SR2)} holds in $\sha$ and \textbf{(SR3)} is a direct consequence of \textbf{(EC4)}. Finally, by \textbf{(EC3)}, we get that $z \ra x \leq (z \ra y) \ra (z \ra (x \wedge y))$ and, in consequence, by \textbf{(EC4)}, we have that
	\[
	(z \ra x) \wedge (z \ra y) \leq  (z \ra y) \wedge ((z \ra y) \ra (z \ra (x \wedge y)) \leq z \ra (x \wedge y).
	\]
	The other inequality in \textbf{(SR4)} follows straightforwardly.
	
	We have seen that every $\{\ra, \wedge\}$-subreduct of a member of Alg$^\ast_{R4^\star}$ is in $\srs$.
	\vskip.2cm
	
	The other inclusion is immediate. It follows applying Lemma \ref{lemant} and Theorem \ref{lemao}
\end{proof}

\vskip.2cm

Let us now note that, writing 1 as a short hand for $\alpha \ra \alpha$, with $\alpha$ some given formula, and 0 as one for $\neg 1$, we have that, for any formula $\beta$:
\begin{itemize}
	\item[] $\vdash 0 \ra \beta$ \ \  and
	\item[] $\neg \beta \dashv \vdash \beta \ra 0$.
\end{itemize}

Let us see that. Let us first prove that $\vdash 0 \ra \beta$:
\vskip.2cm
\begin{tabular}{lll}
	1. & $0 \ra (1 \ra \beta)$ & \textbf{(N1)} \\
	2. & $(1 \ra \beta) \ra \beta$ & theorem of $R4$ \\
	3. & $0 \ra \beta$ & by \textbf{(IL2)}, from 1. and 2. \\
\end{tabular}
\vskip.2cm
Now, we see that $\neg \beta \vdash \beta \ra 0$:
\vskip.2cm
\begin{tabular}{lll}
	1. & $\neg \beta $ & by hypothesis \\
	2. & $\neg \beta \ra (\beta \ra 0)$ & \textbf{(N1)} \\
	3. & $\beta \ra 0$ & by \textbf{(MP)}, from 1. and 2. \\
\end{tabular}
\vskip.2cm
Let us finally see that $\beta \ra 0 \vdash \neg \beta$:
\vskip.2cm
\begin{tabular}{lll}
	1. & $\beta \ra 0$ & hypothesis \\
	2. & $0 \ra \beta$ & theorem of $R4$\\
	3. & $\beta \ra \neg \beta$ & by \textbf{(IL2)}, from 1. and 2.\\
	4. & $(\beta \ra \neg \beta) \ra \neg \beta$ & \textbf{(N2)} \\
	5. & $\neg \beta$ & by \textbf{(MP)}, from 3. and 4. \\
\end{tabular}
\vskip.2cm
As a straightforward consequence, we get the following lemma.
\begin{lem}
	Let $A \in Alg_{R4^\star}^\ast$. The element $0 := \neg 1$ is the bottom of $A$ and the unary operation $\neg$ is $\neg x = x \ra 0$.
\end{lem}

On the other hand, proceeding as in the proof of Lemma \ref{SRSessemiret}, it can be seen that operation $\vee$ makes any $A \in Alg_{R4^\star}^\ast$ a join semilattice.

\begin{lem}
	Let $A \in Alg_{R4^\star}^\ast$, then $(A,\vee, 1)$ is a join semilattice, with $x \vee y = \sup\{x, y\}$, where the supremum is taken with respect to the order in $A$ associated to $\ra$.
\end{lem}

Furthermore, axiom scheme \textbf{(Dist)} implies that $(A,\wedge, \vee, 0, 1)$ is a bounded distributive lattice, and hence, every $A \in Alg_{R4^\star}^\ast$ is a subresiduated lattice. That is to say, $Alg_{R4^\star}^\ast = \srl$.
\vskip.3cm 

From a logical point of view, this last fact implies the following completeness result for calculus $R4^\star$.

\begin{cor}
	For every $\Gamma \cup \{\alpha\} \subseteq \mathrm{Fml_{R4^\star}}$, 	
	$\Gamma \vdash \alpha$ if and only if for every $A \in \srl$ and every $h \in \mathrm{Hom}(\mathrm{Fml_{R4^\star}},A)$, $h(\Gamma) \subseteq \{1\}$ implies $h(\alpha) = 1$.
\end{cor}

\subsection{$IR4^\star$ has the finite model property}
\label{sectionFMP}

Let us show in this subsection that the logic $IR4^\star$ has the finite model property. For that, we shall use that every sub-Hilbert algebra, $A$, is embeddable in a subresiduated lattice, which we shall write $\hat{A}$, through the embedding $j: A \to \hat{A}$, developed in Section \ref{subHilbert}.

\begin{lem}\label{srlfinito}
	Let $L$ be a finite bounded distributive lattice and $D$ a bounded sublattice of $L$. The pair $(L, D)$ forms a subresiduated lattice.
\end{lem}
\begin{proof}
	Take $a, b \in L$. We define $E_{a b} := \{d \in D \ | \ d \wedge a \leq b\}$. Since $0 \wedge a = 0 \leq b$, $E_{a b} \neq \emptyset$. Since $D$ is finite, there exists $u := \bigvee E_{a b}$. Since $L$ is distributive, $a \wedge u = a \wedge \bigvee \{d \ | \ d \in E_{a b}\} = \bigvee \{a \wedge d \ | \ d \in E_{a b}\} \leq b$, because $a \wedge d \leq b$ for every $d \in E_{a b}$. In consequence, $u \in E_{a,b}$, and hence $u = \max E_{a b}$.
\end{proof}

Let $\alpha$ be a formula in the language $\mathcal{L} = \{\ra\}$, and suppose that $\not\vdash_{IR4^\star} \alpha$. Then, $\sha \not\models \alpha = 1$; i.e., there are $A \in \sha$ and a homomorphism $v : Fml_\mathcal{L} \to A$, such that $v(\alpha) \neq 1$.

Let $j: A \to \hat{A}$ be the embedding of $A$ into the subresiduated lattice $\hat{A}$ of Section \ref{subHilbert}. Then, we can extend $v$ to a homomorphism $w : Fml_\mathcal{L} \to \hat{A}$, such that $w(\alpha) \neq 1$ ($w = j \circ v$).

Let us write $Sub(\alpha)$ for the set of subformulas of $\alpha$. By construction, $Sub(\alpha)$ is a finite set. Take $X := w(Sub(\alpha)) = \{w(\beta) \ | \ \beta \in Sub(\alpha)\} \subseteq \hat{A}$. Let $B$ be the bounded sublattice of $\hat{A}$ generated by $X$. As $\hat{A}$ is distributive and $X$ is finite, $B$ is finite. Take $D := B \cap \sq \hat{A} \subseteq B$. By Lemma \ref{srlfinito}, the pair $(B, D)$ defines a subresiduated lattice whose implicative reduct is written $(B, \rab)$.

\begin{lem}\label{lemfmp}
	Let $A$, $\alpha$, $X$ and $B$ be as in previous paragraphs. Then, for $a, b \in X$, if $a \ra b \in X$ then $a \rab b = a \ra b$.
\end{lem}
\begin{proof}
	Since $D \subseteq \sq \hat{A}$, $a \ra b = \max\{d \in \sq \hat{A} \ | \ d \wedge a \leq b\} \geq \max\{d \in D \ | \ d \wedge a \leq b\} = a \rab b$. Since $a \ra b \in D$ and $a \wedge (a \ra b) \leq b$, $a \ra b \leq a \rab b$. Then, $a \ra b = a \rab b$.	
\end{proof}

As a consequence of Lemma \ref{lemfmp}, the unique homomorphism $v': Fml_\mathcal{L} \to B$ such that $v'(p) = w(p)$, on any propositional symbol of the language, satisfies $v'(\alpha) = w(\alpha) = j(v(\alpha)) \neq 1$. We then conclude that $\alpha$ has a finite countermodel. In this way we got a proof of the following result.

\begin{prop}\label{fmpL1}
	The logical system $IR4^\star$ has the finite model property with respect to its algebraic semantics.
\end{prop}
\vskip.3cm

We can adapt the proof of Proposition \ref{fmpL1} and use that $A \in \srs$ can be embedded into a subresiduated lattice (as was seen in Section \ref{semireticulos}) in order to prove the following result.

\begin{prop}\label{fmpL1wedge}
	The $\{\ra, \wedge\}$ fragment of the logic $R4^\star$ has the finite model property with respect to its algebraic semantics.
\end{prop}
\vskip.2cm

\begin{rem}
	It can easily be shown that, in fact, the whole logical system $R4^\star$ has the finite model property with respect to its algebraic semantics. However, we shall get a proof of this fact in next subsection, together with an analogous result for an slightly weaker calculus $R4^\dag$.
\end{rem}

\subsection{The logic $R4^\dag$}
\label{expansionsofL1}

We now introduce another two classes of algebras of interest for the rest of this section.

\begin{defn}\label{srlbs}
	We say that an algebras $(A, \wedge, \vee, \ra, 1)$ of type (2,2,2,0) belongs to the variety $\srs^\vee$ if $(A, \wedge, \ra, 1)$ is a subresiduated semilattice and $(A, \wedge, \vee)$ is a lattice. Equivalently, if $(A, \wedge, \vee)$ is a lattice with last element $1$ satisfying \textbf{(SR1)}-\textbf{(SR4)} of Definition \ref{defdesrs}.
	
	An algebras $(A, \wedge, \vee, \ra, 0, 1)$ of type (2,2,2,0,0) is said to be a \emph{subresiduated lattice in the broad sense} (srlbs for short) if  $(A, \wedge, \vee, \ra, 1) \in \srs^\vee$ and its underlying lattice has a first element, $0$. We write $\srlbs$ for the variety of subresiduated lattices in the broad sense.
\end{defn}

Note that the variety $\srlbs$ has $\srl$ as a (proper) subvariety. Likewise, note that $A \in \srlbs$ has an underlying lattice structure which is not necessarily distributive. This is the essential difference between this variety and $\srl$.

Define the calculus $R4^\dag$ as that which satisfies all the axiom schemes and rules of $R4^\star$ except \textbf{(Dist)}. 

A straightforward computation shows the following result.

\begin{cor}
	Variety $\srlbs$ is the equivalent algebraic semantics of the calculus $R4^\dag$.
\end{cor}

In view of this last result, it seems natural to deeply analyse the variety $\srlbs$. In what follows, we shall see the existence of relevant/interesting examples of srlbs'.

\begin{ex}\label{bv-srlbs}
	Let $L$ be any bounded (not necessarily distributive) lattice. Define on $L$ a binary operation $\ra$ by:
	\begin{equation} \label{b-ra}
		a \ra b =
		\left\{
		\begin{array}{c}
			1, \mathrm{ if }\ a \leq b,\\
			0, \mathrm{ if }\ a \not\leq b\
		\end{array}
		\right.
	\end{equation}
    By reasons which will became clear in next proposition, we call this algebras of type (2,2,2,0,0), 2-subresiduated lattices.
\end{ex}

A straightforward computation shows the following result.

\begin{prop}\label{2srl}
	Let $L$ be any bounded lattice. Consider the algebras $(L, \ra)$ of type (2,2,2,0,0) defined above; i.e., the 2-subresiduated lattices whose underlying lattice is $L$. Then, $(L, \ra)$ satisfies axioms \textbf{(A1)} to \textbf{(A6)} of Definition \ref{srl}.
\end{prop}

\begin{figure}[H]
	\[
	\xymatrix@=.7em{
		& 1 \ar@{-}[ddl] \ar@{-}[dr]& \\
		&                           & c \ar@{-}[dd] \\
		b \ar@{-}[ddr]	&                           &  \\
		&                           & a \ar@{-}[dl] \\
		&0                          &
	}
	\ \ \ \ \ \ \ \ \ \ \ \ \ \ \ \ \ \ \ \ \ \ \ \
	\xymatrix@=.7em{
		& 1 \ar@{-}[ddl] \ar@{-}[ddr] \ar@{-}[dd]& \\
		&                                        &  \\
		a \ar@{-}[ddr]	&  b \ar@{-}[dd]                         &  c \ar@{-}[ddl]\\
		&                                        &  \\
		&0                                       &
	}
	\]
	\caption{The Hasse diagrams of the nonmodular lattice with five elements, $N$ and the modular nondistributive lattice with five elements, $M$.\label{fig1}}
\end{figure}
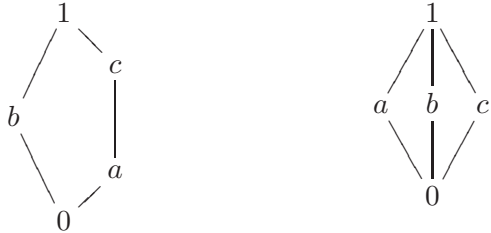

\begin{ex} \label{Db}
		Let $(M, \wedge, \vee, \ra, 0, 1)$ be the algebra of type (2,2,2,0,0) whose underlying lattice structure is that depicted in Figure \ref{fig1} and whose binary operation $\ra$ is presented in the following table.
		\vskip.3cm
\begin{center}		
\begin{tabular}{|c||c|c|c|c|c|}
	\hline
	$\ra$  & 0 & a & b & c & 1 \\
	\hline \hline
	     0 & 1 & 1 & 1 & 1 & 1 \\
	\hline
	     a & b & 1 & b & b & 1 \\
	\hline
	     b & 0 & 0 & 1 & 0 & 1 \\
	\hline
	     c & b & b & b & 1 & 1 \\
	\hline
	     1 & 0 & 0 & b & 0 & 1 \\
	\hline
\end{tabular}
\end{center}	
A direct computation shows that $(M, \ra)$ is a srlbs.
\end{ex}

\begin{ex} \label{Na}
	Let $(N, \wedge, \vee, \ra, 0, 1)$ be the algebra of type (2,2,2,0,0) whose underlying lattice structure is that depicted in Figure \ref{fig1} and whose binary operation $\ra$ is presented in the following table.
		\vskip.3cm
\begin{center}		
	\begin{tabular}{|c||c|c|c|c|c|}
	\hline
	$\ra$ & 0 & a & b & c & 1 \\
	\hline \hline
		0 & 1 & 1 & 1 & 1 & 1 \\
	\hline
		a & 0 & 1 & 0 & 0 & 1 \\
	\hline
		b & a & a & 1 & a & 1 \\
	\hline
		c & 0 & a & 0 & 1 & 1 \\
	\hline
		1 & 0 & a & 0 & a & 1 \\
	\hline
	\end{tabular}
\end{center}	
Another direct computation shows that $(M, \ra)$ is a srlbs.
\end{ex}

The next proposition gives a description of the elements of
$\srlbs$ as pairs $(L, D)$ of a lattice $L$ and an adequate sublattice $D$ of $L$.

\begin{lem}\label{SRSveecomopar}
Let $L$ be a bounded (not necessarily distributive) lattice and $D$ a bounded sublattice of $L$ such that for all $a, b \in L$, the sets $E_{ab} := \{d \in D \ | \ d \wedge a \leq b\}$ have maxima. Then, if we endow $L$ with the binary operation $\ra$, given by $a \ra b := \max E_{ab}$, we have that $(L, \ra) \in \srlbs$.
\end{lem}

\begin{proof}
 Let us check that $(L,\we,\vee,\ra,0,1)$ satisfies the equations defining variety $\srs$. Equations $\mathbf{SL1}$-$\mathbf{SL4}$ hold because $L$ is a bounded lattice.\\
Take $x,y,z\in L$.
\begin{enumerate}
	\item[\textbf{(SR1)}] Since $1\in D$ and $1\we(x\we y)\leq y$, we get $(x\we y)\ra y = 1$.
	\item[\textbf{(SR2)}] Since $(x\ra y)\we z\leq x\ra y$ and $x\ra y\in D$, $x\ra y\leq z\ra(x\ra y)$.
	\item[\textbf{(SR3)}] As $x\ra y = max\{d\in D|d\we x\leq y\}$, it follows that $x\ra y\in\{d\in D|d\we x\leq y\}$ and so $(x\ra y)\we x \leq y$.
	\item[\textbf{(SR4)}] We prove that $z\ra(x\we y)=(z\ra x)\we(z\ra y)$; i.e., that $z\ra(x\we y)=inf\{z\ra x,z\ra y\}$.
	\begin{itemize}
		\item As $z\ra(x\we y)=max\{d\in D|d\we z\leq x\we y\}$, $(z\ra(x\we y))\we z\leq x\we y$ then $(z\ra(x\we y))\we z\leq x$ and $(z\ra(x\we y))\we z\leq y$. In consequence, $(z\ra(x\we y))\we z\leq z\ra x$ and $(z\ra(x\we y))\we z\leq z\ra y$. Hence, $z\ra (x\we y)$ is a lower bound in $D$ of the set $\{z\ra x,z\ra y\}$.
		\item On the other hand, let $t\in D$ be such that $t\leq z\ra x$ and $t\leq z\ra y$. As $(z\ra x)\we z\leq x$ and $(z\ra y)\we z\leq y$, it follows that $t\we z\leq x$ and $t\we z\leq y$. Hence, $t\we z\leq x\we y$ and consequently, $t \leq z\ra(x\we y)$.
	\end{itemize}
\end{enumerate}
\end{proof}

Note that under the hypothesis of Lemma \ref{SRSveecomopar} we have that $D$ is a Heyting algebra and hence a bounded distributive lattice.

Furthermore, a similar argument to that employed in the proof of \cite[Theorem 1]{EH} proves that if $(L, \ra) \in \srlbs$ then $D = \{x \in L \ | \ \Box x = x\}$ is a bounded (distributive) sublattice of $L$ such that for all, $a, b \in L$,  $a \ra b := \max E_{ab}$.
\vskip.3cm

Let us end this subsection analyzing the validity of the finite model property (FMP for short) for logics $R4^\star$ (and hence $R4$) and $R4^\dag$. 
\vskip.2cm 

The proof that $R4^\star$ has the FMP is straightforward. Let $\varphi$ be a formula such that $\not\models \varphi$. Then, there exist $A \in \srl$ and $v : Fml \to A$, a homomorphism from the algebra of formulas such that $v(\varphi) \neq 1$. If $Sub(\varphi)$ is the set of subformulas of $\varphi$, and $X = \{v(\alpha) \ | \ \alpha \in Sub(\varphi)\}$, $X$ is a finite subset of $A$. Take $L$ as the bounded sublattice of $A$ generated by $X$, and $D$ the bounded sublattice of $A$ (and hence of $\sq A$) generated by $\sq X = X \cap \sq A$. We clearly have that $D$ is a bounded sublattice of $L$, which is finite, because $X$ is finite and $A$ is a distributive lattice. Since $L$ is distributive and finite, the pair $(L,D)$ defines a subresiduated lattice $A'$. Write $\rab$ for the implication in $A'$. The following result holds.

\begin{prop}\label{PMFparaR4}
	Let $A$ and $A'$ be as defined in the paragraph above. Then, if $a, b, a \ra b  \in X$, we have that $a \rab b = a \ra b$.
	
	Defining $v': Fml \to A'$ as the unique homomorphism such that $v'(x) = v(x)$ for every variable letter in $Sub(\varphi)$, we get that $v'(\varphi) = v(\varphi) \neq 1$.
\end{prop}

In the case of $R4^\dag$, we have as obstruction for applying the argument above that, in general, $A \in \srlbs$ is not distributive.
\vskip.2cm
Let as consider, as before, that there exist $A \in \srlbs$ and $v : Fml \to A$, a homomorphism from the algebra of formulas, such that $v(\varphi) \neq 1$. Take $Sub(\varphi)$ as the set of subformulas of $\varphi$. The set $X_0 := \{v(\alpha) \ | \ \alpha \in Sub(\varphi)\}$, $X_0$ is a finite subset of $A$. Take $X_0^\sq = X_0 \cap \sq A \subseteq \sq A$ and $D$, the bounded sublattice of $\sq A$ (and hence of $A$) generated by $X_0^\sq$. Since $X_0^\sq$ is finite and $\sq A$ is distributive, $D$ is finite. Then, $X := X_0 \cup D \subseteq A$ is also finite. Let $A'$ be the $\wedge$-subsemilattice of $A$ generated by $X$. Since the variety of semilattices is locally finite and $X$ is finite, $A'$ is finite and in consequence bounded. Then, $A'$ is a lattice (although not necessarily a sublattice of $A$). Let us write $\underline{\vee}$ for the supremum in $A'$. Besides, $D$ is a bounded and distributive sublattice of $A'$.

For every $a, b \in A'$, let us write $E'_{ab} := \{d \in D \ | \ d \wedge a \leq b\}$. Since $A'$ is finite, $u = \underline{\bigvee}E'_{ab}$ exists. Besides, since $E'_{ab} \subseteq D$, $u = \underline{\bigvee}E'_{ab} = \bigvee E'_{ab}$. Since $D \subseteq \sq A$, $E'_{ab} \subseteq E_{ab}$,
where we recall that $E_{ab}$ is defined as $\{d \in \sq A | \ d \wedge a \leq b\}$. Hence,
\begin{equation}\label{E'ab}
\underline{\bigvee}E'_{ab} = \bigvee E'_{ab} \leq \bigvee E_{ab} = \max E_{ab} = a \ra b
\end{equation}

By equation \eqref{E'ab}, $a \wedge u \leq a \wedge (a \ra b) \leq b$; i.e., $u = \max E'_{ab}$. Define on $A'$ a binary operation $\rab$ by $a \rab b := \max E'_{ab}$. By Lemma \ref{SRSveecomopar}, $(A', \rab) \in \srlbs$.

\begin{prop}\label{FMPparaSRSvee}
Let $A$ and $A'$ be as indicated in previous paragraphs. Then,
\begin{enumerate}
	\item if $a,\ b,\ a \ra b \in A'$, then $a \rab b = a \ra b$ \ \ \ and
	\item if $a,\ b,\ a \vee b \in A'$, then $a \underline{\vee} b = a \vee b$.
\end{enumerate}
\end{prop}
\begin{proof}
	As it was already said, $a \rab b = \bigvee E'_{ab} \leq \bigvee E_{ab} = a \ra b$. On the other hand, since $a \ra b \in \bigvee E'_{ab}$, $a \ra b \leq a \rab b$.
	
	Given $a, b \in A'$, $a \underline{\vee} b = \bigwedge\{z \in A' \ | \ a \leq z, b \leq z\} = \min\{z \in A' \ | \ a \leq z, b \leq z\} \geq \min\{z \in A \ | \ a \leq z, b \leq z\} = a \vee b$. Furthermore, $a \vee b \in \{z \in A' \ | \ a \leq z, b \leq z\}$ and, in consequence, $a \underline{\vee} b \leq a \vee b$.	
\end{proof}

As a consequence of Proposition \ref{FMPparaSRSvee}, we get that if we consider the homomorphism $v':Fml \to A'$, induced by $v$, then $v'(\varphi) = v(\varphi) \neq 1$. Hence, there is a finite countermodel for $\varphi$.
\vskip.2cm

Finally, from Propositions \ref{PMFparaR4} and \ref{FMPparaSRSvee}, it follows that $R4^\dag$ has the finite model property.

\subsection{On sub-Hilbert lattices}
\label{otherexpansion}

In \cite{CFMSM} the following class of algebras containing both the varieties $\srl$ and that of Hilbert lattices is described. Here, we give another description of this variety (see Appendix \ref{appen} for details).

\begin{defn}\label{sHS}
	A \textbf{sub-Hilbert lattice} is an algebra $(A, \wedge, \vee, \ra, 1)$ of type (2,2,2,0) such that:
	\begin{enumerate}
		\item $(A, \wedge, \vee, 1)$ is a lattice with last element and
		\item The following identities hold:
		\begin{enumerate}
			\item[\textbf{(SH1)}] $(x \wedge y) \ra y = 1$,
			\item[\textbf{(SH2)}] $x \wedge (x \ra y) \leq y$,
			\item[\textbf{(SH3)}] $x \ra y \leq (y \ra z) \ra (x \ra z)$ \ \ and
			\item[\textbf{(SH4)}] $x \ra (y \ra z) \leq (x \ra y) \ra  (x \ra z)$.
		\end{enumerate}	
	\end{enumerate}
\end{defn}

We say that an algebra $(A, \wedge, \vee, \ra, 0, 1)$ is a \textbf{bounded sub-Hilbert lattice} if it is a sub-Hilbert lattice with first element 0.

Write $\shs$ for the variety of bounded sub-Hilbert lattices.
\vskip.2cm

Note that in any $A \in \shs$, $x \leq y$ if and only if $x \ra y = 1$. Indeed, if $x \ra y = 1$, then $x = x \wedge 1 = x \wedge (x \ra y) \leq y$. On the other hand, if $x \leq y$, then $x \wedge y = x$ and hence, $x \ra y = (x \wedge y) \ra y = 1$.

Besides, note that \textbf{(SH4)} guarantees the monotony of $\ra$ in is second coordinate because if $x \leq y$ (and hence $x \ra y = 1$) then, $1 = z \ra 1 = z \ra (x \ra y) \leq (z \ra x) \ra (z \ra y)$. Hence, $z \ra x \leq z \ra y$.

Finally, note that if $x \leq y$ we get, using \textbf{(SH3)}, that $1 \leq (y \ra z) \ra (x \ra z)$; i.e., 
$y\ra z \leq x\ra z$. Thus, $\ra$ is antimonotone in the first coordinate.
\vskip.2cm

The following useful properties are shared by all the elements of $\shs$.
\begin{lem}
	Let $A \in \shs$ and $x, y, z \in A$. Then, the following inequalities are satisfied in $A$:
	\begin{enumerate}
		\item $x \ra (y \wedge z) \leq x \ra y$,
		\item $(x \ra y) \wedge (y \ra z) \leq x \ra z$ \ \ and
		\item $x \ra y \leq z \ra (x \ra y)$.
	\end{enumerate}
\end{lem}
\begin{proof}
	\item 1. It straightforwardly follows from the antimonotony of $\ra$.
	\item 2. It is a direct consequence of \textbf{(SH3)} and \textbf{(SH2)}.
	\item 3. It follows from \textbf{(SH3)}. More concretely, we have that 
	$(x \ra y) \leq (y \ra y) \ra (x \ra y) = 1 \ra (x \ra y)$.
	Since $1\ra (x\ra y) \leq z\ra (x\ra y)$ then $x\ra y \leq z\ra (x\ra y)$.
\end{proof}

\begin{ex}\label{notSRS}
	Consider the lattice $B_2$ depicted below, seen as a bounded lattice.
	
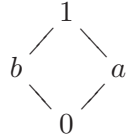
\[
	\xymatrix@=.7em{
		& 1 \ar@{-}[dl] \ar@{-}[dr]&  \\
		b \ar@{-}[dr]	&           & a \ar@{-}[dl]  \\
		&0                          &
	}
	\]
	Let us endow it with the binary operation\footnote{This implication turns $B_2$ into an order Hilbert algebra with supremum and infimum \cite{CSM} and hence into a bounded sub-Hilbert lattice.}
	\[
	x \ra y =
	\left\{
	\begin{array}{c}
		1, \mathrm{ if }\ x \leq y,\\
		y, \mathrm{ if }\ x \not\leq y\
	\end{array}
	\right.
	\]
	
	The algebra $(B_2, \wedge, \vee, \ra, 0, 1)$ is an element of $\shs$ which does not satisfy \textbf{(SR4)}\footnote{We have $a \ra (a \wedge b) \neq (a \ra a) \wedge (a \ra b)$.} of Definition \ref{defdesrs}. Hence, it is not a subresiduated lattice (even in the broad sense). Then, $\shs \supsetneqq \srlbs$.
\end{ex}

\subsection{Weakening $R4^\star$ a little further}
\label{otherexpansion}

From a logical point of view, we can endow the implicative fragment of $R4$ with a ''conjunction'', by means of axioms  $\textbf{(C1)}$, $\textbf{(C2)}$ and $\textbf{(C3)}$, with $\textbf{(C1)}$ to $\textbf{(C3)}$ the usual intuitionistic axioms for conjunction. Alternatively, we could have defined the connective $\wedge$ by means of the axioms $\textbf{(C1)}$, $\textbf{(C2)}$ and a (weaker) rule in place of axiom \textbf{(C3)}.

\begin{defn} \label{defL1+}
	The system $R4^{+}$ has as axiom schemes those of $R4^\dag$, except \textbf{(EC3)} and besides the rules \textbf{(MP)} and \textbf{(T)} of $R4^\dag$, the rule
\[
\dfrac{\delta \ra \alpha, \ \delta \ra \beta}{\delta \ra (\alpha \wedge \beta)}\ \ \mathbf{(C)}
\]
\end{defn}

Clearly, scheme \textbf{(C3)} implies the rule \textbf{(C)}. That's why system $R4^+$ is, in principle, a weakening of system $R4^\dag$. The aim of this subsection is the study of system $R4^+$.

A similar argument to that employed in the case of $R4^\dag$ shows that rule \textbf{(C4)}, is also derivable in $R4^{+}$.
\vskip.3cm

Let us begin checking that system $R4^{+}$ is also implicative and, hence, algebraizable. In order to do that, we check that \textbf{(IL3$_\wedge$)} still holds for $R4^+$.
\vskip.2cm

Let $\alpha$ and $\beta$ be formulas. The following is a proof in $R4^+$.
\vskip.2cm
\begin{tabular}{lll}
	1. & $\alpha \ra \beta$ & by hypothesis \\
	2. & $(\alpha \wedge \delta) \ra \alpha$ & \textbf{(C1)} \\
	3. & $(\alpha \wedge \delta) \ra \delta$ & \textbf{(C2)} \\
	4. & $(\alpha \wedge \delta) \ra \beta$ & from 2. and 1. by \textbf{(IL2)} \\
	5. & $(\alpha \wedge \delta) \ra (\beta \wedge \delta)$ & from 3. and 4. by rule \textbf{(C)}.
\end{tabular}
\vskip.2cm
Then, $R4^{+}$ is algebraizable. Let us call Alg$^\ast_+$ its associated quasivariety. It has as a base of quasi-equations, quasi-equations \textbf{(B)}, \textbf{(I)}, \textbf{(A)}, \textbf{(T)}, \textbf{(S)}, \textbf{(EC1)}, \textbf{(EC2)}, \textbf{(ED1)}, \textbf{(ED2)}, \textbf{(ED3)}, \textbf{(EN1)}, \textbf{(EN2)} and
\begin{enumerate}
	\item[\textbf{(QC)}] If $z \ra x = 1$ and $z \ra y = 1$, then $z \ra (x \wedge y) = 1$.
\end{enumerate}

Note that in this quasivariety $\wedge$ is still the infimum with respect to the order induced by $\ra$. On one hand, we have that $x \wedge y \leq x, y$. On the other, if $z \leq x, y$, then $z \ra x = z \ra y = 1$. By \textbf{(QC)}, we get that $z \ra (x \wedge y)$; i.e., $z \leq x \wedge y$.
\vskip.2cm

Besides, it also holds in Alg$^\ast_+$ 
\begin{enumerate}
	\item[\textbf{(EC4)}] $(x \wedge (x \ra y)) \ra y = 1$.
\end{enumerate} 
It may be alternatively written as $x \wedge (x \ra y) \leq y$.

\begin{thm}
	Quasivarieties $\shs$ and Alg$^\ast_+$ agree.
\end{thm}
\begin{proof}
	Let $A \in$ Alg$^\ast_+$. As we have seen, $(A, \wedge, \vee, 0, 1)$ is a bounded lattice. Equation \textbf{(SH1)} coincides with \textbf{(EC1)} and \textbf{(SH2)} is \textbf{(EC4)}. Equations \textbf{(SH3)} and \textbf{(SH4)} are identities satisfied by any bounded sub-Hilbert algebra and $A$ is one of them. Thus, we get that $A \in \shs$.
	\vskip.2cm
	Conversely, assume that $A \in \shs$. Let us check that the equations and quasi-equations defining Alg$^\ast_+$ hold in $A$.
	\vskip.2cm
	
	First, note that \textbf{(EC1)} is \textbf{(SH1)} and \textbf{(I)} and \textbf{(T)} are consequence of \textbf{(SH1)} together with the facts that $x \wedge x = x$ and $x \wedge 1 = x$. Equation \textbf{(EC2)} follows from \textbf{(SH1)} and the fact that $x \wedge y = y \wedge x$.
	
	The validity of \textbf{(A)} follows from the facts that $A \in \shs$ and $x \leq y$ if and only if $x \ra y = 1$.
	\vskip.2cm
    Equation \textbf{(B)} follows from \textbf{(SH3)}.
    \vskip.2cm

    Equation \textbf{(S)} is equation \textbf{(SH4)}.
    \vskip.2cm

    Finally, assume that $z \ra x = z \ra y = 1$; i.e., that $z \leq x, y$. Since $\wedge$ is the infimum in $A$,  $z \leq x \wedge y$; i.e., \textbf{(QC)} holds in $A$.
\end{proof}

Example \ref{notSRS} shows that $\shs$ does not satisfy the following equation:
\[
((z \ra x) \wedge (z \ra y)) \ra (z \ra (x \wedge y)) = 1,
\]
 meanwhile this equation is satisfied in $\srs$ (and hence in $\srlbs$). Then, due to the soundness and completeness of the logics $R4^+$ and $R4^\dag$ with respect to their algebraic semantics, we obtain the following result.

\begin{cor}
	Systems $R4^+$ and $R4^\dag$ present different logics.
\end{cor}

\section*{Acknowledgments}
This work was supported by Consejo Nacional de Investigaciones
Cient\'ificas y T\'ecnicas (CONICET-Argentina), Universidad
Nacional de La Plata [PPID/X047] and by Universidad Nacional de San Juan (Subvention F1191-Cicitca).

\appendix
\section{Sub-Hilbert lattices}
\label{appen}
\vskip.3cm
In \cite{CFMSM} a variety also named sub-Hilbert lattices was introduced. The purpose of this appendix is to show that Definition \ref{sHS} gives an alternative equational bases for the variety of sub-Hilbert lattices, as it was defined in \cite{CFMSM}.

\begin{defn}[Definition 3.3 of \cite{CFMSM}] \label{defshrl}
An algebra $(A,\we,\vee,\ra,1)$ of type $(2,2,2,0)$ is a sub-Hilbert lattice
if $(A,\we,\vee,1)$ is a lattice, $(A,\we,\ra,1)$ is a hemi-implicative semilattice and for every
$a, b, c, d \in A$, the following inequalities are satisfied:
\begin{enumerate}[\normalfont (a)]
\item $a\ra b \leq c\ra (a\ra b)$,
\item $(a\vee b)\ra c \leq (a\ra c)\we (b\ra c)$,
\item $a\ra (b\we c) \leq (a\ra b)\we (a\ra c)$,
\item $d\ra (a\ra (b\ra c)) \leq (d\ra (a\ra b)) \ra (d\ra (a\ra c))$,
\item $\sq a \ra (\sq b \ra \sq c) = \sq b \ra (\sq a\ra \sq c)$,
\item $\sq a\ra (\sq b \ra \sq c) = (\sq a \ra \sq b) \ra (\sq a\ra \sq c)$.
\end{enumerate}

Here $\sq a$ is shothand for $1 \ra a$, as usual.
\end{defn}

Let us now see that (a)-(f) gives an alternative equational bases for $\shs$.
\vskip.3cm

Let $A\in \shs$. First note that $\square(A)$ is a Hilbert algebra. Indeed,
it follows from $\textbf{(SH1)}$ that $1\in \square(A)$. For every $a, b, c\in A$, it
follows from $\textbf{(SH4)}$ that $(a\ra (b\ra c)) \ra ((a\ra b) \ra (a\ra c)) = 1$.
Finally let $a,b \in \square(A)$ such that $a\ra b = b \ra a = 1$. Equation \textbf{(SH2)}
allow us to show that $a = b$. Thus, $\square(A)$ is a Hilbert algebra. Hence, the equations (e) and (f)
are satisfied.
The conditions (b) and (c) are the antimonotony and the monotony of the implication
respectively, which was proved. Condition (a) was also proved. Finally, we will see (d).
Let $a,b,c,d \in A$. By $\textbf{(SH4)}$,
\begin{equation}\label{aeq1}
d\ra ((a\ra b) \ra (a\ra c)) \leq (d \ra (a\ra b)) \ra (d \ra (a\ra c)),
\end{equation}
\begin{equation}\label{aeq2}
a\ra (b\ra c) \leq (a\ra b) \ra (a\ra c).
\end{equation}
It follows from \eqref{aeq2} and (c) that
\begin{equation} \label{aequ3}
d\ra (a\ra (b\ra c)) \leq d\ra ((a\ra b) \ra (a\ra c)).
\end{equation}
Thus, it follows from \eqref{aeq1} and \eqref{aequ3} that
\[
d\ra (a\ra (b\ra c)) \leq (d \ra (a\ra b)) \ra (d \ra (a\ra c)).
\]
Thus, condition (d) is satisfied. Then, $A\in \shrl$.
\vskip.2cm

Conversely, assume that $A$ satisfies (a)-(f) in Definition \ref{defshrl}. Note that for every $a,b\in A$, $\square(a\ra b) = a\ra b$.
Condition $\textbf{(SH1)}$ follows from the fact that $a\leq b$ if and only if $a\ra b = 1$.
Condition $\textbf{(SH2)}$ follows by definition. Equation
$\textbf{(SH4)}$ follows from (d) by considering $d = 1$. Finally we will see $\textbf{(SH4)}$.
Finally we will see $\textbf{(SH3)}$. Let $a,b,c \in A$.
Then $a\ra (b\ra c) \leq (a\ra b)\ra (a\ra c)$. Hence,
\[
(b\ra c) \ra (a\ra (b\ra c)) \leq (b\ra c) \ra ((a\ra b)\ra (a\ra c)).
\]
By (e), we get 
\[
(b\ra c) \ra ((a\ra b)\ra (a\ra c)) = (a\ra b)\ra ((b\ra c) \ra (a\ra c)),
\]
so
\[
(b\ra c) \ra (a\ra (b\ra c)) \leq (a\ra b)\ra ((b\ra c) \ra (a\ra c)).
\]
Also by (e), we have that $(b\ra c) \ra (a\ra (b\ra c)) = 1$, so
\[
a\ra b \leq (b\ra c) \ra (a\ra c).
\]
We have proved $\textbf{(SH3)}$. Therefore, $A\in \shs$.

\end{document}